\documentclass[12pt]{article}
\usepackage{amsfonts}
\usepackage{epsfig}
\usepackage{graphics}
\usepackage{amsmath,amsthm}
\usepackage{amssymb}
\usepackage{inputenc}
\usepackage{color}
\def \ve{\varepsilon}
\def \be{\begin{equation}}
\def \ee{\end{equation}}

\theoremstyle{plain}
\newtheorem{thm}{Theorem}

\newtheorem{lem}[thm]{Lemma}

\font \twbbb= msbm10 scaled \magstep0                 

\font \tenbbb= msbm7 scaled \magstep0                 

\newfam\bbbfam             

\textfont\bbbfam=\twbbb \scriptfont\bbbfam=\tenbbb    %

\begin{document}

\footnotetext{Mathematics Subject Classification 2000: }

\numberwithin{equation}{section}
\theoremstyle{plain}
\newtheorem{Th}{Theorem}
\newtheorem{Lm}{Lemma}
\newtheorem*{Df*}{Definition}
\newtheorem{Cr}{Corollary}[section]
\renewcommand{\abstractname}{}

\author{
\centerline{\bf A.~Piatnitski} \\  \centerline{Narvik University College, Postboks 385, 8505 Narvik, Norway,}\\
\centerline{P. N. Lebedev Physical Institute of RAS, 53, Leninski pr., Moscow 119991, Russia,}\\
\centerline{e-mail: andrey@sci.lebedev.ru} \and \centerline{\bf
A.~Rybalko}
\\ \centerline{Simon Kuznets Kharkiv National University of Economics,}\\
\centerline{9a Lenin ave.,
Kharkiv 61166, Ukraine,} \\
\centerline{e-mail: n\_rybalko@yahoo.com}
\and \centerline{\bf V.~Rybalko} \\
\centerline{Mathematical Department, B.Verkin Institute for Low Temperature Physics} \\
\centerline{and Engineering of the NASU,
47 Lenin ave., Kharkiv 61103, Ukraine,} \\
\centerline{e-mail: vrybalko@ilt.kharkov.ua}}

\title{Singularly perturbed spectral problems with Neumann boundary conditions}

\maketitle

\begin{abstract}
The paper deals with the Neumann spectral problem for a singularly perturbed second order elliptic
operator with bounded lower order terms. The main goal is to provide a refined description of the
limit behaviour of the principal eigenvalue and eigenfunction. Using the 
logarithmic transformation we reduce the studied problem to additive eigenvalue problem for 
a singularly perturbed Hamilton-Jacobi equation. Then assuming that the Aubry set of the Hamiltonian 
consists of a finite number of points or limit cycles situated in the domain or on its boundary, 
we find the limit of the eigenvalue and formulate the selection criterium that allows us to choose a solution
of the limit Hamilton-Jacobi equation which gives the logarithmic asymptotics of the principal eigenfunction. 
\end{abstract}


\section{Introduction}
This paper is devoted to the asymptotic analysis of the first eigenpair for singularly perturbed spectral problem,
depending on the small parameter $\ve>0$, for the elliptic equation
\begin{equation}
\label{Pr}
\ve a_{ij} (x)\frac{\partial^2 u_\ve} {\partial
x_i\partial x_j} +b_{i} (x)\frac{\partial
u_\ve}{\partial x_i} +c(x) u_\ve=\lambda_\ve u
\end{equation}
in a smooth bounded domain $\Omega \subset \mathbb{R}^N$ with the  boundary condition
\begin{equation}
\frac{\partial u_\ve}{\partial \nu} = 0,
\label{Bound}
\end{equation}
on $\partial\Omega$, where $\frac{\partial}{\partial \nu}$ denotes derivative with
respect to the external normal.

The bottom of the spectrum of elliptic operators plays a crucial role in many applications. In particular, the first eigenvalue
and the corresponding eigenfunction of (\ref{Pr})--(\ref{Bound}), are important in understanding the large-time
behavior of the underlying non-stationary convection-diffusion model with reflecting boundary.
Due to the Krein-Rutman theorem the first eigenvalue $\lambda_\ve$ of (\ref{Pr})--(\ref{Bound})
(the eigenvalue with the maximal real part) is simple and real, the corresponding eigenfunction $u_\ve$
can be chosen to satisfy $u_\ve (x)>0$ in $\Omega$.

The goal of this work is to study the asymptotic behavior of $\lambda_\ve$ and $u_\ve$ as
$\ve\to 0$. While in the case of constant function $c(x)$ in (\ref{Pr}) the first eigenpair
is (trivially) explicitly found, the asymptotic behavior of the first eigenpair is quite nontrivial
when $c(x)$ is a nonconstant function, in particular, the eigenfunction might exhibit an exponential
localization.

Boundary-value problems for singularly perturbed elliptic operators have been actively studied starting
from 1950s. We mention here a pioneering work \cite{ViLu}, where for a wide class of
operators (so-called regularly degenerated operators) the asymptotics of solutions were obtained.

In the works \cite{We72}, \cite{We75}, \cite{Fri73} (see also \cite{WF}) the principal eigenvalue of singularly perturbed convection-diffusion equations with the
Dirichlet boundary condition was investigated by means of large deviation techniques for diffusion processes
with small diffusion. In \cite{DeFr} the estimates for the principal eigenvalue were obtained by comparison
arguments and elliptic techniques.

The case when convection vector field has a finite number of hyperbolic equilibrium points and cycles was
studied in \cite{Kif1_80} where methods of dynamical systems are combined with those of stochastic differential equations. These results were generalized in \cite{Ki87} to the case when the boundary of domain is  invariant
with respect to convection vector field. Similar problem in the presence of zero order term was considered in
\cite{Ki90}.

In \cite{Pe90} the viscosity solutions techniques for singularly perturbed Hamilton-Jacobi equation were used in order to study the principal eigenfunction of the  adjoint Neumann convection-diffusion
problem. The logarithmic asymptotics of the eigenfunction was constructed.

The work \cite{Pi} deals with the principal eigenpair of operators with a large zero order term on a compact Riemannian manifold. The approach developed in this work is based on large deviation and variational techniques.

Dirichlet spectral problem for a singularly perturbed operators with rapidly oscillating locally periodic coefficients
was studied in \cite{PR} and \cite{PRR}. In \cite{PR} with the help of viscosity solutions method the limit of the principal eigenvalue and the logarithmic asymptotics of the principal eigenfunction were found. These asymptotics were improved in \cite{PR} and \cite{PRR} using the blow up analysis.

In the present work when studying problem \eqref{Pr}--\eqref{Bound}, we make use of the standard viscosity solutions techniques in order to obtain the logarithmic asymptotics of the principal eigenfunction. However, the limit Hamilton-Jacobi equation in general is not uniquely solvable and does not give information about the limit behaviour of
$\lambda_\ve$. Therefore, we have to consider higher order approximations in (\ref{Pr})--\eqref{Bound}. Under rather general assumptions on the structure of the Aubry set of the limit Hamiltonian, we find the limit of $\lambda_\ve$ and
can choose the solution of the limit problem which determines the asymptotics of the principal eigenfunction.
Notice that we did not succeed to make the blow up analysis work in the case under consideration. In this case, for components of the Aubry set located on the boundary,  the natural rescaling still leads to a singularly perturbed operators . Instead, we study a refined structure of solutions of the limit Hamilton-Jacobi equation in the vicinity of the Aubry set. This allows us to construct test functions that satisfy the perturbed equation up to higher order.

We also would like to remark that, with obvious modifications, the results of this work as well as the developed techniques remain valid for the boundary condition of the form
$$
\frac{\partial u_\ve}{\partial \beta}=0,
$$
where $\beta$ is a $C^2$-smooth vector field on $\partial\Omega$ non-tangential at any point of  $\partial\Omega$. In particular, conormal vector field  $\beta_i=a_{ij}\nu_j$ can be considered.

\section{Problem setup and results}

We study problem (\ref{Pr})--(\ref{Bound}) under the following assumptions on the operator coefficients and the domain:
\begin{itemize}

\item [({\bf a1})] $\Omega$ is a bounded domain in $\mathbb R^N$, $N\geq 2$, with a $C^2$ boundary;
\item [({\bf a2})] all the coefficients are $C^2$-functions in $\overline\Omega$;
\item [({\bf a3})] the matrix $(a_{ij})$ is symmetric and uniformly elliptic.
\end{itemize}
Further assumptions on the vector field $b$ will be formulated later on.

\bigskip
Since $u_\ve>0$ in $\Omega$
we can represent $u_\ve$ in the form
$$
u_\ve=e^{-W_\ve (x)/\ve},
$$
this results in the following nonlinear PDE
\begin{equation}
\label{Pr1}
- a_{ij} (x)\frac{\partial^2 W_\ve} {\partial
x_i\partial x_j} +\frac{1}{\ve} H(\nabla W_\ve, x) +c(x) =\lambda_\ve
\qquad \text{in} \ \Omega
\end{equation}
or
\begin{equation}
\label{PrII}
- \ve a_{ij} (x)\frac{\partial^2 W_\ve} {\partial
x_i\partial x_j} + H(\nabla W_\ve, x) +\ve c(x) =\ve\lambda_\ve
\qquad \text{in} \ \Omega
\end{equation}
with the boundary condition
\begin{equation}
\label{Bound1}
\frac{\partial W_\ve}{\partial \nu} = 0
\qquad \text{on} \ \partial\Omega,
\end{equation}
where
\begin{equation}
\label{Hamilt}
H(p,x)=a_{ij} (x) p_i p_j- b_{i} (x)p_i
\end{equation}
is a function to be referred to as a Hamiltonian. Passing to the limit as $\ve\to 0$ in \eqref{PrII}, with the
help of the standard approach based on the maximum principle, we
can show that $W_\ve$ converges uniformly (up to extracting a
subsequence) to a viscosity solution $W(x)$ of the
Hamiltion-Jacobi equation
\begin{equation}
\label{HamEq}
 H(\nabla W(x), x) =0
\qquad \text{in} \ \Omega
\end{equation}
with the boundary condition
\begin{equation}
\label{HamBc}
\frac{\partial W}{\partial \nu} = 0
\qquad \text{on} \ \partial\Omega.
\end{equation}

Recall that a function $W\in C(\overline{\Omega})$ is called a viscosity solution of equation
 (\ref{HamEq}) if for every test function $\Phi\in C^{\infty}(\overline{\Omega})$ the following holds
\begin{itemize}
\item if $W-\Phi$ attains a maximum at a point $\xi\in\Omega$ then $W(\nabla \Phi (\xi),\xi) \leq 0$;
\item if $W-\Phi$ attains a minimum at $\xi\in\Omega$ then $W(\nabla \Phi (\xi),\xi) \geq 0$.
\end{itemize}
The boundary condition (\ref{HamBc}) is understood in the following sense, $\forall \Phi\in C^{\infty}(\overline{\Omega})$
\begin{itemize}
\item if $W-\Phi$ attains a maximum at $\xi\in\partial\Omega$ then
$\min \Bigl\{ H(\nabla \Phi (\xi),\xi), \frac{\partial\Phi}{\partial\nu} (\xi) \Bigr\} \leq 0$;
\item if $W-\Phi$ attains a minimum at $\xi\in\partial\Omega$ then
$\max\Bigl\{ H(\nabla \Phi (\xi),\xi), \frac{\partial\Phi}{\partial\nu} (\xi) \Bigr\} \geq 0$.
\end{itemize}

It is known  \cite{I} that every solution of problem
(\ref{HamEq})--(\ref{HamBc}) has the representation
\begin{equation}
\label{MainWRepresent}
W(x)=\inf_{y\in\mathcal{A}_H}\Bigl\{ d_{H} (x,y)+W(y) \Bigr\},
\end{equation}
where $\mathcal{A}_H$  is so-called Aubry set and $ d_{H} (x,y)$ is a distance function. To define $\mathcal{A}_H$ and
 $ d_{H} (x,y)$ consider solutions of Skorohod problem
\begin{equation}
\begin{cases}
\label{Skor}
 \eta(t)\in \overline{\Omega}, t\geq 0 \\
\dot \eta(t)+\alpha (t) \nu (\eta (t))=v(t)  \qquad \text{with} \
\alpha (t)\geq 0 \  \  \text{and} \ \alpha (t)= 0 \ \text{when}
\ \eta(t)\notin\partial\Omega\\
\eta(0)=x
\end{cases}
\end{equation}
where $v\in L^1((0,\infty); \mathbb{R}^N)$ is a given vector field and $x\in\overline{\Omega}$ is a given initial point,
while the curve 
$\eta\in W_{loc}^{1,1}((0,\infty); \mathbb{R}^N)$
and the function $\alpha\in L^1((0,\infty); \mathbb{R}_+)$ are unknowns.
Under our standing assumptions on $\Omega \ (\partial \Omega\in C^2)$    Skorohod problem
 (\ref{Skor}) has a solution, see \cite{I}.

Consider now the Legendre transform $L(v,x)=\sup\limits_{p\in\mathbb{R}^N} (v\cdot p - H(p,x))$
and define the distance function
\begin{equation}
\label{dfunction} d_H (x,y)=\inf \Bigl\{  \int_0^t L(-v(s),\eta (s)
)\,{\rm d}s,\ \ \eta \ \text{solves} \   (\ref{Skor}), \ \eta(0)=x,\
\eta(t)=y, \ t>0 \Bigr\}.
\end{equation}
Next we recall the variational definition of the Aubry set
\begin{equation}
\label{AubryNeumann} x\in \mathcal{A}_H \Longleftrightarrow
\forall \delta
>0 \ \ \inf \Bigl\{  \int_0^t L( -v(s), \eta (s) )\, {\rm d}s,\ \ \eta \
\text{solves}  \ (\ref{Skor}), \ \eta(0)=\eta(t)=x, \ t>\delta
\Bigr\}=0.
\end{equation}

In this work we assume that the Aubry set has finite number of connected components
\begin{multline}
\label{H1}
\mathcal{A}_H=  \bigcup_{\rm finite} \mathcal{A}_k \ and \ each\ \mathcal{A}_k \ is\
either\ isolated \ point\\  or \ closed\ curve\ lying\
entirely\ either\ in\ \Omega\ or\ on\ \partial\Omega.
\end{multline}
Additionally we assume that
\begin{multline}
\label{H2}
if \ \mathcal{A}_k\subset\Omega \ then\ \mathcal{A}_k \ is\ either\ hyperbolic \
fixed\ point\\  or \ hyperbolic \ limit\ cycle\ of\ the\ ODE\ \dot x=b(x);
\end{multline}

\begin{multline}
\label{H3}
if \ \mathcal{A}_k\subset\partial\Omega \ then\ the\ normal\ component \ b_\nu (x)  \ of \ the\ field\ b(x)\
\ is\ strictly \ positive\ on\ \mathcal{A}_k \\
\ and\ \mathcal{A}_k \ is\ either\ hyperbolic \ fixed\ point\  or \ hyperbolic\
limit\ cycle\ of\ the\ ODE\\ \dot x=b_\tau(x)
\ on \ \partial\Omega,\
where\ b_\tau(x) \ denotes\ the\ tangential\ component\ of\ b(x)\ on\ \partial\Omega.
\end{multline}

In order to state the main result of this work we assign to each
component $\mathcal{A}_k $ of $\mathcal{A}_H$ a number $\sigma
(\mathcal{A}_k) $  as follows. If $\mathcal{A}_k $ is a fixed
point $\{\xi\}$ of the ODE $\dot x=b(x)$ and $\xi\in\Omega$,
linearizing the ODE near $\xi$ to get $\dot z=B(\xi)z$ we define
$\sigma (\mathcal{A}_k) $ by
\begin{equation}
\label{Fpoint}
\sigma (\mathcal{A}_k) =-\sum\limits_{\theta_i>0}\theta_i+c(\xi),
\end{equation}
where $\theta_i$ are the real parts of eigenvalues of the
matrix $B(\xi)$. Note that the hyperbolicity of the fixed point
means that the eigenvalues of $B(\xi)$ cannot have zero real part.
If $\mathcal{A}_k=\{\xi\} $ and $\xi\in\partial\Omega$, consider
the ODE $\dot x=b_{\tau}(x)$ on $\partial\Omega$ in a neighborhood
of the point $\xi$. Passing to the linearized ODE $\dot
z=B_{\tau}(\xi)z$ in the tangent plane to  $\partial\Omega$  at
the point $\xi$, we denote by $\tilde\theta_i$ the  real
parts of the eigenvalues of $B_{\tau}(\xi)$ and set
\begin{equation}
\label{Fpoint1}
\sigma (\mathcal{A}_k) =-\sum\limits_{\tilde\theta_i>0}\tilde\theta_i+c(\xi),
\end{equation}

Consider now the case when $\{\mathcal{A}_k\}\subset\Omega$ is a
limit cycle of ODE $\dot x=b(x)$. Let $P>0$ be the minimal period
of the cycle and let $\Theta_i$ be the absolute values of
eigenvalues of the linearized Poincar\'{e} map.
 (Recall that the limit cycle is said hyperbolic if there
are no eigenvalues of linearized  Poincar\'{e} map with absolute
value equal to 1.) We define now $\sigma (\mathcal{A}_k) $ by
setting
\begin{equation}
\label{Lcycle}
\sigma (\mathcal{A}_k) =-\frac{1}{P}\sum\limits_{\Theta_i>1}\log\Theta_i+\frac{1}{P}\int_{0}^{P}c(\xi(t)){\rm d}t,
\end{equation}
where $\xi(t)$ solves $\dot \xi=b(\xi)$ and $\xi(t)\in \mathcal{A}_k$.

Finally,  in the case when $b_{\nu}>0$ on $\mathcal{A}_k$ and $\mathcal{A}_k$ is a limit cycle of the ODE
$\dot x=b_{\tau}(x)$ on $\partial\Omega$, we set
\begin{equation}
\label{Lcycle1}
\sigma (\mathcal{A}_k) =-\frac{1}{P}\sum\limits_{\widetilde\Theta_i>1}
\log\widetilde\Theta_i+\frac{1}{P}\int_{0}^{P}c(\xi(t)){\rm d}t,
\end{equation}
where $\dot \xi=b_{\tau}(\xi)$ and $\xi(t)\in \mathcal{A}_k$, $P$ is the minimal period and $\widetilde\Theta_i$
are the absolute values of the eigenvalues of the linearized Poincar\'{e} map. 

The main result of this work is
\begin{thm}\label{Mainth}
Let conditions ({\bf a1})--({\bf a3}) be fulfilled, and assume that the Aubry set $\mathcal{A}_H$ satisfies
  (\ref{H1}), (\ref{H2}) and  (\ref{H3}). Then the
first eigenvalue $\lambda_\ve$ of (\ref{Pr}) converges as $\ve\to 0$ to
\begin{equation}
\label{Max} \lim_{\ve\to 0}\lambda_\ve=\max\Bigl\{\sigma(
\mathcal{A}_k);  \mathcal{A}_k\subset  \mathcal{A}_H\Bigr\},
\end{equation}
where $\sigma( \mathcal{A}_k)$ is given by  (\ref{Fpoint}) or
(\ref{Fpoint1}) if $\mathcal{A}_k$ is a fixed point
in $\Omega$ or  on $\partial\Omega$, and $\sigma( \mathcal{A}_k)$
is defined by (\ref{Lcycle}) or (\ref{Lcycle1}) if
$\mathcal{A}_k$ is a limit cycle  in $\Omega$ or  on
$\partial\Omega$. Moreover, if the maximum in   (\ref{Max}) is
attained at exactly one component $\mathcal{M}:=\mathcal{A}_{k_0}$, then
the scaled logarithmic transform $w_\ve=-\ve\log u_\ve$ of the
first eigenfunction $u_\ve$ (normalized by max $u_\ve=1$)
converges uniformly in $\overline{\Omega}$ to the maximal
viscosity solution $W$ of (\ref{HamEq})-(\ref{HamBc}) vanishing on
$\mathcal{M}$, i.e. $W(x)=d_H(x,\mathcal{M}).$
\end{thm}

\section{Passing to the limit by vanishing viscosity techniques}
\label{SECTIONvanvisclimit} In this section we pass to the limit, as
$\ve\to0$,
in  equation
(\ref{PrII}) and boundary
condition (\ref{Bound1}) to get problem
(\ref{HamEq})--(\ref{HamBc}). We use the standard technique based
on the maximum principle and the a priori uniform $W^{1,\infty}$
bound for $W_\ve$ obtained by Berstain's method
\cite{LU}, \cite{S}.

First, considering  (\ref{PrII}) at the maximum and minimum points of $W_\ve(x)$ we easily get

\begin{lem}
\label{Leigval}
The first eigenvalue $\lambda_\ve$ satisfies the estimates $\min c(x) \leq\lambda_\ve\leq\max c(x)$.
\end{lem}

Next we establish the $W^{1,\infty}$ bound for $W_\ve$ in

\begin{lem}
\label{Leigfun}
Let $u_\ve$ be normalized by $\max u_\ve=1$ ( i.e. $\min W_\ve=0$). Then $\| W_\ve\|_{W^{1,\infty}(\Omega)}\leq C$ with a constant
C independent of $\ve$.
\end{lem}
\begin{proof}
Following \cite{L} observe that the boundary condition
$\frac{\partial W_{\ve}}{\partial\nu}=0$ yields the pointwise
bound
$$
\frac{\partial}{\partial\nu} |\nabla W_\ve |^2 \leq C  |\nabla W_\ve |^2
\qquad \text{on} \ \partial\Omega.
$$
Therefore, for an appropriate positive function $\phi\in C^2(\overline{\Omega})$,
\begin{equation}
\label{Bcaux}
\frac{\partial}{\partial\nu} \biggl(\phi|\nabla W_\ve |^2\biggr) \leq -  \biggl(\phi |\nabla W_\ve |^2\biggr)
\qquad \text{on} \ \partial\Omega.
\end{equation}
Next we use Bernstain's method to obtain a uniform bound for
$\omega_\ve(x)=\phi(x) |\nabla W_\ve(x) |^2$,
following closely the line of \cite{EI}, Lemma 1.2. In view of (\ref{Bcaux}) either $|\nabla W_\ve | \equiv 0$ and we
have nothing  to prove,  or $\max w_\ve$ is attained at a point $\xi\in\Omega$. In the latter case we have
$\nabla \omega_\ve(\xi)=0$ and
$$
 a_{ij} \frac{\partial^2}{\partial x_i\partial x_j} \biggl(   \phi |\nabla W_\ve |^2  \biggr)\leq 0
\qquad \text{at} \ x=\xi.
$$
Expanding the left hand side of this inequality we get
\begin{equation}
\label{Au1}
2\ve  \phi a_{ij} \frac{\partial^2 W_\ve }{\partial x_i\partial x_k}\frac{\partial^2 W_\ve } {\partial x_j\partial x_k} \leq
-2\ve  \phi a_{ij} \frac{\partial^3 W_\ve }{\partial x_i\partial x_j\partial x_k}\frac{\partial W_\ve }{\partial x_k}
-4\ve  a_{ij} \frac{\partial\phi }{\partial x_j} \frac{\partial^2 W_\ve }{\partial x_i\partial x_k}\frac{\partial W_\ve }{\partial x_k}
-\ve  a_{ij} \frac{\partial^2\phi }{\partial x_i\partial x_j}
 |\nabla W_\ve |^2.
\end{equation}
Using (\ref{PrII}) we obtain
\begin{equation}
\label{Au2}
-\ve  \phi a_{ij} \frac{\partial^3 W_\ve } {\partial x_i\partial x_j\partial x_k}\frac{\partial W_\ve } {\partial x_k} \leq
\ve \phi \frac{\partial a_{ij} }{\partial x_k} \frac{\partial^2 W_\ve }{\partial x_i\partial x_j}\frac{\partial W_\ve }{\partial x_k}+
C \biggl( \omega_\ve^{3/2}+\omega_\ve+\omega_\ve^{1/2}+1\biggr) ,
\qquad \text{at} \ x=\xi,
\end{equation}
where we have also exploited the fact that $\nabla \omega_\ve(\xi)=0$. Substitute now (\ref{Au2}) into (\ref{Au1}) to derive
\begin{equation}
\label{Fin1}
\ve  \frac{\partial^2 W_\ve } {\partial x_i\partial x_k} \frac{\partial^2 W_\ve } {\partial x_i\partial x_k} \leq
C \biggl( \omega_\ve^{3/2}+1\biggr) ,
\qquad \text{at} \ x=\xi.
\end{equation}
On the other hand it follows from (\ref{PrII}) that
\begin{equation}
\label{Fin2}
\omega_\ve \leq
C \Biggl( \ve\sum \Biggl| \frac{\partial^2 W_\ve }{\partial x_i\partial x_j}\Biggr| +1\Biggr) .
\end{equation}
Combining (\ref{Fin1}) and  (\ref{Fin2}) we obtain  $\omega_\ve \leq  C$ and the required uniform bound follows.
\end{proof}

With a priori bounds from Lemma \ref{Leigval} and Lemma \ref{Leigfun} it is quite standard to  pass to the limit in (\ref{PrII}). Indeed,
up to extracting a subsequence, $W_\ve\to W$ uniformly in $\overline{\Omega}$ and     $\lambda_\ve\to \lambda$. Consider
a test function $\Phi\in C^\infty(\overline{\Omega})$ and assume that $W-\Phi$ attains strict maximum at a point $\xi$. Then
$W_\ve-\Phi$ attains local maximum at $\xi_\ve$ such that  $\xi_\ve \to \xi$ as $\ve\to 0$. If $\xi_\ve\in\Omega$ then
$\nabla W_\ve(\xi_\ve)=\nabla \Phi(\xi_\ve)$ and
$$
 a_{ij} \frac{\partial^2 W_\ve}{\partial x_i\partial x_j} (\xi_\ve)\leq a_{ij} \frac{\partial^2 \Phi}{\partial x_i\partial x_j} (\xi_\ve)
\ \ \text{if}\ \xi_\ve\in\Omega
$$
and $\frac{\partial \Phi}{\partial \nu} (\xi_\ve)\leq 0$ if $\xi_\ve\in\partial\Omega$.
Passing to the limit as $\ve\to 0$ and using (\ref{PrII}) and Lemma \ref{Leigval} we get
$$
H ( \nabla\Phi(\xi),\xi)\leq 0\ \ \text{if} \ \xi\in\Omega,  \ \  \ \
\text{and}\ \min
\Biggl\{ H ( \nabla\Phi(\xi),\xi), \frac{\partial\Phi}{\partial\nu}(\xi)\Biggr\}\leq 0\ \  \hbox{if }\xi\in\partial \Omega.
$$
Arguing similarly in the case when $\xi$ is a strict minimum of $W-\Phi$ we conclude that $W$ is a
viscosity solution
of (\ref{HamEq})-(\ref{HamBc}).

\section{Matching  lower and upper bounds for eigenvalues and
selection of the solution of  (\ref{HamEq})--(\ref{HamBc})}

Due to the results of Section \ref{SECTIONvanvisclimit} we can
assume, passing to a subsequence if necessary, that  eigenvalues
$\lambda_\ve$ converge to a finite limit $\lambda$ and
functions $W_\ve$ converge uniformly in $\overline \Omega$ to a
solution $W$ of problem (\ref{HamEq})--(\ref{HamBc}) as $\ve\to 0$.
In the following four steps we prove that $\lambda$ and $W(x)$ are
described by Theorem \ref{Mainth}.

\smallskip
%
\noindent {\it Step I: Significant component(s) of ${\mathcal A}_H$.} Recall
the definition of the partial order relation $\preceq$ on
$\mathcal{A}_H$ introduced in \cite{PR} as follows
\begin{equation}
\label{Order}
 \mathcal{A}^\prime\preceq  \mathcal{A}^{\prime\prime} \ \Longleftrightarrow \ W( \mathcal{A''})=d_H( \mathcal{A''}, \mathcal{A'})+W( \mathcal{A'}).
\end{equation}
Since the distance function $d_H(x,y)$ satisfies the triangle inequality and
$d_H( \mathcal{A''}, \mathcal{A'})+d_H( \mathcal{A'}, \mathcal{A''})>0$ for different components $\mathcal{A'}, \mathcal{A''}$
 of the Aubry set $ \mathcal{A}_H$, (\ref{Order}) indeed defines the partial order relation.

Condition (\ref{H1}) assumes that there are finitely many different components of the Aubry set.
It follows that there exists at least
one minimal component  $ \mathcal{M}:=\mathcal{A}_{k_0}$ (such that, $\forall \mathcal{A}_k\not=\mathcal{M}$,
 either $ \mathcal{M} \preceq \mathcal{A}_k$ or $ \mathcal{M}$ and $\mathcal{A}_k$ are not comparable).

Now show that
\begin{equation}
\label{Sign} W(x)=d_H(x,  \mathcal{M})+W( \mathcal{M}) \ \text{in}\ U\cap\overline{\Omega},\
\text{where}\ U\ \text{is a neighborhood of} \ \mathcal{M}.
\end{equation}
Indeed, otherwise there is a sequence $x_i\to \mathcal{M} $ and a component $\mathcal{A}_k\not=\mathcal{M}$
such that $W(x_i)=d_H(x_i,\mathcal{A}_k)+W(\mathcal{A}_k)$. Then taking the limit we derive
$W(\mathcal{M})=d_H(\mathcal{M},\mathcal{A}_k)+W(\mathcal{A}_k)$,
that is $ \mathcal{A}_k\preceq  \mathcal{M}$ which contradicts
the minimality of $\mathcal{M}$.

In what follows a component $\mathcal{M}$ such that (\ref{Sign})
is satisfied will be called a significant component. We
have shown that under condition (\ref{H1}) there is at least one significant
component in the Aubry set $\mathcal{A}_H$.

\smallskip

\noindent {\it Step II: Upper bound for eigenvalues.} The crucial
technical result in the proof of Theorem \ref{Mainth} is the
following Lemma whose proof is presented in subsequent four
Sections dealing separately with four possible cases of the structure of $\mathcal{M}$.

\begin{lem}
\label{MainTechLem} Let $\mathcal{M}$ be a significant component
of the Aubry set $\mathcal{A}_H$ satisfying either (\ref{H2}) or
(\ref{H3}). Then for sufficiently small $\delta>0$ there are continuous
functions $W^{\pm}_\delta(x)$, $W^{\pm}_{\delta,\ve}(x)$ and
neighborhoods $U_\delta$  of $\mathcal{M}$ such that
\begin{equation}
\label{StricMinMax}
W^{\pm}_\delta(\mathcal{M})=0 \quad
\text{and}\quad
W_\delta^{-}(x)<W(x)-W(\mathcal{M})<W_\delta^{+}(x)
\quad\text{in}\quad U_\delta\cap\overline\Omega\setminus
\mathcal{M},
\end{equation}
$W^{\pm}_{\delta,\ve}\in C^2(U_\delta\cap\overline \Omega)$,
$W^{\pm}_{\delta,\ve}\to W^{\pm}_\delta$ uniformly in
$U_\delta\cap\overline \Omega$ as $\ve\to 0$, and
\begin{equation}
 \liminf_{\delta\to 0}\liminf_{\ve\to 0, \, \xi_\ve\to
 \mathcal{M}}
\Bigl(-a_{ij}(\xi_\ve)\frac{\partial^2
W^{+}_{\delta,\ve}}{\partial x_i\partial x_j}(\xi_\ve)  +
\frac{1}{\ve}H(\nabla
W^{+}_{\delta,\ve}(\xi_\ve),\xi_\ve)+c(\xi_\ve)\Bigr)\geq
\sigma(\mathcal{M}). \label{forLowerBound}
\end{equation}
\begin{equation}
\limsup_{\delta\to 0}\limsup_{\ve\to 0, \, \xi_\ve\to
 \mathcal{M}}
\Bigl(-a_{ij}(\xi_\ve)\frac{\partial^2
W^{-}_{\delta,\ve}}{\partial x_i\partial x_j}(\xi_\ve)  +
\frac{1}{\ve}H(\nabla
W^{-}_{\delta,\ve}(\xi_\ve),\xi_\ve)+c(\xi_\ve)\Bigr)\leq
\sigma(\mathcal{M}). \label{forUpperBound}
\end{equation}
Moreover, if $U_\delta\cap\partial\Omega\not=\emptyset$ then the
functions $W^{\pm}_{\delta,\ve}$ also satisfy
$\frac{\partial W^{+}_{\delta,\ve}}{\partial \nu}>0$ on
$U_\delta\cap\partial\Omega$, and $\frac{\partial
W^{-}_{\delta,\ve}}{\partial\nu}<0$ on
$U_\delta\cap\partial\Omega$.
\end{lem}

Now, assuming that we know a minimal component $\mathcal{M}$ of
the Aubry set $\mathcal{A}_H$, we can identify the limit $\lambda$
of eigenvalues $\lambda_\ve$. Consider the difference
$W_\ve-W^{-}_{\delta,\ve}$, where $W^{-}_{\delta,\ve}$ are test
functions described in Lemma \ref{MainTechLem}. By
(\ref{StricMinMax}) the function $W-W^{-}_{\delta}-W(\mathcal{M})$ vanishes on
 $\mathcal{M}$ while it is strictly positive in a punctured neighborhood of $\mathcal{M}$.
 Then, since
$W_\ve-W^{-}_{\delta,\ve}$ converge uniformly to
$W-W^{-}_{\delta}$ as $\ve\to 0$ in a neighborhood of
$\mathcal{M}$, there exists a sequence of local minima $\xi_\ve$
of $W_\ve-W^{-}_{\delta,\ve}$ such that $\xi_\ve\to\mathcal{M}$.
Moreover, if $\mathcal{M}\cap \partial \Omega\not=\emptyset$ then
$\frac{\partial W^{-}_{\delta,\ve}}{\partial\nu}<\frac{\partial
W_{\ve}}{\partial\nu}=0$ on $\partial\Omega$ (locally near
$\mathcal{M}$) and therefore $\xi_\ve\in\Omega$ for sufficiently
small $\ve$. For such $\ve$ we have
$$
\nabla W_\ve= \nabla W_{\delta,\ve}^- \quad and\quad
-a_{ij}\frac{\partial^2 W_\ve}{\partial x_i \partial x_j}\leq
-a_{ij}\frac{\partial^2 W_{\delta,\ve}^-}{\partial x_i \partial x_j}
 \ \ at\ \  x=\xi_\ve.
$$
Therefore,
$$
\begin{aligned}
\lambda_\ve &=- a_{ij}(\xi_\ve)\frac{\partial^2 W_\ve}{\partial
x_i \partial x_j}(\xi_\ve)
+\frac{1}{\ve}H(\nabla W_\ve(\xi_\ve), \xi_\ve)+c(\xi_\ve)\\
&\leq-a_{ij}(\xi_\ve)\frac{\partial^2 W^{-}_{\delta,\ve}}{\partial
x_i
\partial x_j}(\xi_\ve) +\frac{1}{\ve}H(\nabla W^{-}_{\delta,\ve}(\xi_\ve), \xi_\ve)
+c(\xi_\ve).
\end{aligned}
$$
Thus we can use (\ref{forUpperBound}) here to pass first to the
$\limsup$ as $\ve\to 0$ and then as $\delta\to 0$, this yields
$\limsup_{\ve\to 0} \lambda_\ve \leq \sigma(\mathcal{M})$.
Similarly one obtains the matching upper bound so that
\begin{equation}
\label{Lsup1} \lim_{\ve\to 0} \lambda_\ve = \sigma(\mathcal{M}).
\end{equation}
However, since at this point $\mathcal{M}$ is unknown (it depends
on $W$ and thus on the particular choice of a subsequence made in
the beginning of the Section) equality (\ref{Lsup1}) guarantees
only the upper bound
\begin{equation}
\label{Lsup2} \limsup_{\ve\to 0} \lambda_\ve \leq
\max\Bigl\{\sigma( \mathcal{A}_k);  \mathcal{A}_k\subset
\mathcal{A}_H\Bigr\},
\end{equation}
where the $\limsup_{\ve\to 0}$ is taken over the whole family $\{\lambda_\ve,\,
\ve>0\}$.

\smallskip

\noindent {\it Step III: Lower bound for eigenvalues.}
Consider a component $\mathcal A$ of the Aubry set $\mathcal{A}_H$
such that $\sigma(\mathcal{A})=\max\{\sigma(\mathcal{A}_k);\, \mathcal{A}_k\subset \mathcal{A}_H\}$.
Introduce a smooth
function $\rho(x)$ such that
$$
\rho(x)\geq 0\ \text{in}\ \overline{\Omega},\ \ \rho(x)=0\
\text{in a neighborhood of } \ \mathcal{A}, \ \text{and}\
\rho(x)>0, \ \text{when} \ x\in \mathcal{A}_H\setminus \mathcal{A}
$$
 and consider the first eigenvalue
$\overline \lambda_\ve$ of an auxiliary eigenvalue problem
\begin{equation}
\label{AuxPr}
\ve a_{ij} (x)\frac{\partial^2 \overline{u}_\ve}
{\partial x_i\partial x_j} +b_{i} (x)\frac{\partial
\overline{u}_\ve}{\partial x_i}
+\Bigl(c(x)-\frac{1}{\ve}\rho(x)\Bigr)
\overline{u}_\ve=\overline{\lambda}_\ve \overline{u}_\ve \
\text{in}\ \Omega,
\end{equation}
with the Neumann condition $\frac{\partial
\overline{u}_\ve}{\partial \nu}=0$ on $\partial \Omega$. By the
Krein-Rutman theorem the eigenvalue $\overline{\lambda}_\ve$ is
real and of multiplicity one, and $\overline{u}_\ve$ being
normalized by $\max_{\Omega}\overline{u}_\ve=1$ satisfies
$\overline{u}_\ve>0$ in $\Omega$. Note that the adjoint problem also has a
sign preserving eigenfunction.
Then it follows that
\begin{equation}
\overline{\lambda}_\ve\leq \lambda_\ve.
\label{lowereigenvaluebounnd}
\end{equation}
Indeed, otherwise we have
\begin{equation}
\ve a_{ij} (x)\frac{\partial^2 {u}_\ve} {\partial x_i\partial x_j}
+b_{i} (x)\frac{\partial {u}_\ve}{\partial x_i}
+\Bigl(c(x)-\frac{1}{\ve}\rho(x)\Bigr)
{u}_\ve-\overline{\lambda}_\ve
{u}_\ve=-\Bigl(\overline{\lambda}_\ve-\lambda_\ve+\frac{1}{\ve}\rho(x)\Bigr)
{u}_\ve<0 \quad \text{in}\ \Omega. \label{AuxOperInequal}
\end{equation}
On the other hand, by Fredholm's theorem the right hand side in
(\ref{AuxOperInequal}) must be orthogonal (in $L^2(\Omega)$) to
any eigenfunction of the problem adjoint to (\ref{AuxPr}). Since
the latter problem has a sign preserving eigenfunction we arrive at
a contradiction which proves (\ref{lowereigenvaluebounnd}).

Let $\overline{W}_\ve:=-\ve\log \overline{u}_\ve$ be the scaled
logarithmic transform of $\overline{u}_\ve$, i.e.
$\overline{u}_\ve=e^{-\overline{W}_\ve/\ve}$. Following the line
of Section \ref{SECTIONvanvisclimit} one can show that, up to
extracting a subsequence, functions $\overline{W}_\ve$ converge
(uniformly in $\overline{\Omega}$) to a viscosity solution $\overline{W}$of
the problem
\begin{equation}
\label{modifiedHanJac}
H(\nabla\overline{W}(x),x)-\rho(x)=\Lambda \
\text{in}\ \Omega
\end{equation}
with the boundary condition $\frac{\partial \overline{W}}{\partial
\nu} =0$, where $\Lambda=\lim_{\ve\to 0} \ve\overline{\lambda}_\ve$. Note
that the argument in Lemma~\ref{Leigval} yields now bounds of the
form $-\frac{C}{\ve}\leq \overline{\lambda}_\ve\leq C$ with some
$C>0$ independent of $\ve$.  Nevertheless  these bounds are sufficient
to derive problem (\ref{modifiedHanJac}) with the Neumann boundary condition.
 Moreover, since
$\rho=0$ in a neighborhood of $\mathcal{A}$ one can show that
$\Lambda=0$ using testing curves $\eta$ from (\ref{AubryNeumann})
in the variational representation for the additive eigenvalue
$\Lambda$ (see \cite{I}),
$$
\Lambda=-\lim_{T\to\infty}\inf\Bigl\{\frac{1}{T}\int_0^T
\bigl(L(-v,\eta)+\rho(\eta)\bigr)\,{\rm d} t;\, \eta\ \text{solves
(\ref{Skor}) with} \ \eta(0)=x\in\overline{\Omega}\Bigr\}.
$$
This implies, in particular,  that
$$
\overline{W}(x)=d_{H}(x,\mathcal{A}) \ \text{in a neighborhood
of}\ \mathcal{A},
$$
where $d_{H}(x,y)$ is the distance function given by
(\ref{dfunction}). Then arguing as in
second step
we obtain
$$
\overline{\lambda}_\ve\to \sigma(\mathcal{A}).
$$
Thanks to (\ref{lowereigenvaluebounnd}) this yields the lower
bound $ \liminf \lambda_\ve\geq \max
\bigl\{\sigma(\mathcal{A}_k);\,
\mathcal{A}_k\in\mathcal{A}_H\bigr\}$ complementary
to  (\ref{Lsup2}). Thus formula (\ref{Max}) is proved.

\smallskip

\noindent {\it Step IV: Selection of the solution of  (\ref{HamEq})-(\ref{HamBc}).}
Let us assume now that the maximum in  (\ref{Max}) is attained at exactly one
component $\mathcal{M}$. Then comparing (\ref{Max}) with (\ref{Lsup1}) we see
that $\mathcal{M}$ is the unique significant component in $\mathcal{A}_H$, therefore it
is the only minimal component of $\mathcal{A}_H$ with respect to the order relation $\prec$.
Thus   $\mathcal{M}$ is the least component  of $\mathcal{A}_H$. It follows
that $W(\mathcal{A}_k)-W(\mathcal{M})=d_H(\mathcal{A}_k,\mathcal{M})$
for every $\mathcal{A}_k\subset \mathcal{A}_H$. Then by (\ref{MainWRepresent})
the representation  $W(x)=d_H(x,\mathcal{M})+W(\mathcal{M})$ holds
in $\overline{\Omega}$. Finally, since
$\min_{\overline{\Omega}} W(x)=\lim_{\ve\to 0} \min_{\overline{\Omega}} W_\ve(x)=0$
we have $W(\mathcal{M})=0$, i.e. $W(x)=d_H(x,\mathcal{M})$.

\smallskip

\noindent
Theorem \ref{Mainth} is proved. \hfill $\square$

\section{Construction of test functions: case of fixed points in $\Omega$}
\label{FixP}
The central part in the proof of Theorem \ref{Mainth} is the construction of test functions satisfying
the conditions of Lemma \ref{MainTechLem} for different types of components of the Aubry set $\mathcal{A}_H$.
Consider first the case when a fixed point $\xi\in\Omega$ of the ODE $\dot x=b(x)$ is a significant component
of
$\mathcal{A}_H$.
We can assume that $W(\xi)=0$, subtracting an appropriate constant if necessary.
Then $W(x)$ is given by
\begin{equation}
\label{Ident} W(x)=d_H(x,\xi) \ \text{in a neighborhood} \ U(\xi)
\ {\rm of}\ \xi.
\end{equation}
We begin by studying the local behavior of $W(x)$ near $\xi$.
Consider for sufficiently small $z$ the ansatz
\begin{equation}
\label{Repr} W(z+\xi)=\Gamma_{ij} z_i z_j+o(|z|^2)
\end{equation}
with a symmetric $N\times N$ matrix $\Gamma$. After substituting (\ref{Repr}) into (\ref{HamEq})
we are led to the
Riccati matrix equation
\begin{equation}
\label{Ric}
4\Gamma Q\Gamma-\Gamma B-B^*\Gamma =0,
\end{equation}
where $Q=\Bigl(a_{ij}(\xi)\Bigr)_{i,j=\overline{1,N}}$,
$B=\Bigl(\frac{\partial b_i}{\partial
x_j}(\xi)\Bigr)_{i,j=\overline{1,N}}$.

Next we show that
(\ref{Repr}) holds with $\Gamma$ being the maximal symmetric
solution of (\ref{Ric}); for existence of such a solution see, e.g., \cite{LR}  or  \cite{AFIJ} .
To this end
consider the
solution $D$ of the Lyapunov matrix equation
\begin{equation}
\label{Lyap}
D(4\Gamma Q-B)+(4 \Gamma Q-B)^*D =2I
\end{equation}
given by
\begin{equation}
\label{LyapFormulaIntegral}
D=2\int_{-\infty}^0   e^{(4\Gamma Q-B)^*t} e^{(4Q\Gamma -B)t}{\rm d}t.
\end{equation}
By Theorem 9.1.3 in  \cite{LR} all the eigenvalues of the matrix
$4Q\Gamma-B$ have positive real parts, so that the integral in (\ref{LyapFormulaIntegral}) does converge.
Set
$$
\Gamma_\delta^{\pm}=\Gamma\pm\delta D.
$$
 Then $\Gamma_\delta^-$
satisfies
\begin{equation}
\label{Matsubs}
4\Gamma_\delta^-Q\Gamma_\delta^--\Gamma_\delta^-B-B^*\Gamma_\delta^-\leq -\delta I
\end{equation}
for sufficiently small $\delta>0$.

Introduce the quadratic function  $W_\delta^-(x):= \Gamma_\delta^-(x-\xi)\cdot
(x-\xi)$. Thanks to (\ref{Matsubs}) this function
satisfies
\begin{equation}
\label{HJacSubsol}
H(\nabla W_\delta^-(x),x)\leq -\frac{\delta}{2}|x-\xi|^2 \quad {\rm  in  \ a \ neighborhood\  of\ } \xi.
\end{equation}
This yields the following result whose proof is identical to the
proof of Lemma 16 in \cite{PR}  (see also the arguments in the proof of Lemma \ref{Lcompar} below).

\begin{lem}
\label{Lsubs} The strict pointwise inequality  $W_\delta^-(x)<
W(x)$ holds in a punctured neighborhood of $\xi$ for sufficiently
small $\delta>0$.
\end{lem}

\noindent
Next consider the function  $W_\delta^+(x):= \Gamma_\delta^+(x-\xi)\cdot (x-\xi)$.

\begin{lem}
\label{Lsupers}
The strict pointwise inequality $ W_\delta^+(x)> W(x)$ holds in a punctured neighborhood of $\xi$ for sufficiently
small $\delta>0$.
\end{lem}

\begin{proof}
According to (\ref{Ident}), the following inequality holds
$$
W(x)\leq \int_0^t L(-v(\tau),\xi+\eta(\tau)){\rm d}\tau
$$
for every control  $v(\tau)$  such that the solution of the ODE
$$
\dot\eta(\tau)=v(\tau), \quad  \eta(0)=z:=x-\xi
$$
vanishes at the final time $t$ and remains in a small neighborhood
of 0 for any $0\leq\tau\leq t$. We can take the final time
$t=+\infty$ and construct $v(\tau)$ by setting
$v(\tau)=-(4Q\Gamma-B)\eta(\tau)$, where $\eta(\tau)$ in the
solution of the ODE
$$
\dot\eta=-(4Q\Gamma-B)\eta,\quad \eta(0)=z.
$$
As already mentioned (see Theorem 9.1.3. in  \cite{LR}) all the eigenvalues of
the matrix $4Q\Gamma-B$ have positive real parts, therefore
$|\eta(\tau)|\leq C |z|$ and $\eta(\tau)\to 0$ as $\tau\to
+\infty$. Moreover,  the latter convergence is exponentially fast.

Thus we have
$$
\begin{aligned}
L(-v(\tau), \eta(\tau)+\xi)&=\frac{1}{4}a^{ij}(\xi+\eta)(-\dot\eta_i+b_i(\eta))(-\dot\eta_i+b_i(\eta))=\\
&=\frac{1}{4}a^{ij}(\xi)(-\dot\eta_i+B_{ik}\eta_k)(-\dot\eta_j+B_{jl}\eta_l)+O(|\eta|^3),
\end{aligned}
$$
where $a^{ij}(x)$ denote the entries of the matrix inverse to
$\bigl( a_{ij}(x) \bigr)_{i,j=\overline{1,N}}$. Next recall that
$Q_{ij}=a_{ij}(\xi)$ and that $\Gamma$ solves (\ref{Ric}). Taking
this into account we obtain
$$
\begin{aligned}
\int_0^\infty L(\eta+\xi, -v(\tau))&=\frac{1}{4}\int_0^\infty a^{ij}(\xi)(-\dot\eta_i+B_{ik}\eta_k)(-\dot\eta_j+B_{jl}\eta_l)+O(|z|^3)=\\
&=-2\int_0^\infty \Gamma\eta\cdot \dot\eta {\rm d}\tau + \int_0^\infty \Gamma\eta\cdot (\dot\eta +B\eta) {\rm d}\tau +O(|z|^3)=\\
&=\Gamma z\cdot z  + \int_0^\infty\eta\cdot(-4\Gamma Q\Gamma+\Gamma B+B^*\Gamma)\eta {\rm d}\tau+O(|z|^3)=\\
&=\Gamma_{ij}z_i z_j +O(|z|^3).
\end{aligned}
$$
Finally, since by the definition of $ \Gamma_\delta^+$, $ \Gamma_\delta^+=\Gamma+\delta D$ with $D>0$, then
for sufficiently small $z\not=0$
we have
$$
W(z+\xi)\leq\Gamma z\cdot z +O(|z|^3)< \Gamma_\delta^+ z\cdot z.
$$
\end{proof}

Lemmas \ref{Lsubs} and \ref{Lsupers} show that functions $W^{\pm}_\delta$ do satisfy
conditions of Lemma \ref{MainTechLem}. To complete the proof
of Lemma  \ref{MainTechLem} in the case of $\mathcal{M}$ being a fixed point in $\Omega$
we define functions $W^{\pm}_{\delta,\ve}$ simply by setting $W^{\pm}_{\delta,\ve}:=W^{\pm}_\delta$.
Thanks to (\ref{HJacSubsol}) we have
\begin{equation}
\label{UpperBoundVerify}
\limsup_{\ve\to 0}\Bigl( a_{ij}(\xi_\ve)\frac{\partial^2 W_{\delta,\ve}^-}{\partial x_i \partial x_j}(\xi_\ve)
+\frac{1}{\ve}H(\nabla W_{\delta,\ve}^-(\xi_\ve), \xi_\ve)
+c(\xi_\ve)\Bigr)\leq-2a_{ij}(\xi)(\Gamma_{ij}-\delta D_{ij}) +c(\xi),
\end{equation}
as soon as $\xi_\ve \to \xi$ when $\ve\to 0$.
According to  Proposition 20 in \cite{PR},  $-2a_{ij}(\xi)\Gamma_{ij} +c(\xi)=\sigma(\{\xi\})$,
thus (\ref{UpperBoundVerify}) yields (\ref{forUpperBound}).
Similarly one verifies that $W_{\delta,\ve}^+$ satisfies (\ref{forLowerBound}).

\section{Construction of test functions: case of fixed points on $\partial \Omega$}
\label{FP}

Consider now the case of significant component of the Aubry set  $\mathcal{A}_H$ being a
hyperbolic fixed point $\xi$ of
the ODE $\dot x=b_\tau (x)$ on $\partial \Omega$, where $b_\tau (x)$ denotes the tangential
component of the
vector field $b(x)$ on
$\partial \Omega$. As above, without loss of generality, we assume that $W(\xi)=0$.

It is convenient to introduce local coordinates near $\partial \Omega$ so that
$x=X(z_1,\dots, z_N)$ with $z_N=z_N(x)$ being the distance from
$x$ to $\partial\Omega$ ($z_N(x)>0$ if $x\in\Omega$) and
$z^\prime=(z_1,...,z_{N-1})$ representing coordinates  on $\partial \Omega$ in a neighborhood of the
point $\xi$. The latter coordinates are chosen so that the map
$X(z^\prime,z_N)$ is $C^2$-smooth and $z^ \prime(\xi)=0$.
Moreover, the matrix $\Bigl(\frac{\partial X_i}{\partial
z_j}\Bigr)_{i,j=\overline{1,N}}$ is orthogonal when $z^\prime=0$
and $z_N=0$ (at the point $\xi$).  In these new variables
equations (\ref{HamEq}) and (\ref{PrII}) read
\begin{equation}
\label{nHamJac} S(\nabla_z W,z)=0
\end{equation}
and
\begin{equation}
\label{nperturbed} -\ve a_{ij} \left(X(z)\right) \mathcal{T}^{-1}_{ki} (z)
\frac{\partial} {\partial z_k}\biggl( \mathcal{T}^{-1}_{lj}(z)\frac{\partial
W_\ve} {\partial z_l}\biggr)+S(\nabla_z W_\ve,z)
=\ve\bigl(\lambda_\ve - c(X(z))\bigr),
\end{equation}
where
$$
S(p,z)= a_{ij} (X(z)) \mathcal{T}^{-1}_{ki}(z) \mathcal{T}^{-1}_{lj}(z) p_k
p_l-b_i(X(z)) \mathcal{T}^{-1}_{ki}(z) p_k
$$
and $\Bigl(\mathcal{T}^{-1}_{ij}(z)\Bigr)_{i,j=\overline{1,N}}$ is the
inverse matrix to $\Bigl(  \frac{\partial X_i} {\partial
z_j}(z)\Bigr)_{i,j=\overline{1,N}}$. Note that according to
hypothesis (\ref{H3})
\begin{equation}
\label{H3newcoord} b_i(X(z))\mathcal{T}^{-1}_{Ni}(z)<0 \ \ \text{for
sufficiently small}\ |z|.
\end{equation}


Like in Section \ref{FixP} we construct the leading term of the
asymptotic expansion of $W$ near the fixed point $\xi$ in the form
of a quadratic function. Taking into account the boundary
condition $\frac{\partial W}{\partial z_N}=0$ (that is
(\ref{HamBc}) rewritten in aforementioned local coordinates) we
write down the following ansatz
$$
W(X(z^\prime,z_N)) = \tilde\Gamma_{ij}z_i^\prime z_j^\prime
+o(|z|^2+z_N^2).
$$
with a symmetric $(N-1)\times (N-1)$ matrix $\widetilde\Gamma$ satisfying
the Riccati equation
\begin{equation}
\label{Ric1} 4\tilde \Gamma \tilde Q \tilde \Gamma-\tilde \Gamma \tilde B
-{\tilde B}^*\tilde\Gamma=0,
\end{equation}
where $\tilde Q=\Bigl(  a_{ij}(\xi) \mathcal{T}^{-1}_{ki}(0)\mathcal{T}^{-1}_{lj}(0)
\Bigr)_{k,l=\overline{1,N-1}}$ and $ \tilde B=\Bigl(
\mathcal{T}^{-1}_{ki}(0)\frac{\partial b_i} {\partial x_j}\bigl(\xi\bigr)
 \frac{\partial X_j} {\partial z_l}(0)
 \Bigr)_{k,l=\overline{1,N-1}}$. Note that  $\tilde B$ is
 nothing but the matrix in the ODE $\dot z^\prime= \tilde B z^\prime$ obtained by
linearizing the ODE $\dot x=b_\tau(x)$ near $\xi$  in the local
coordinates $z^\prime=(z_1^\prime,\dots,z_{N-1}^\prime)$ on
$\partial \Omega$.

Let $\tilde \Gamma$ be the maximal symmetric solution of (\ref{Ric1}),
and let $\tilde D$ be a solution of the Lyapunov matrix
equation
\begin{equation}
\label{Lyap1}
\tilde D (4\tilde \Gamma \tilde Q - \tilde B)+ (4\tilde \Gamma \tilde Q - \tilde B)^* \tilde D=2I.
\end{equation}
By Theorem 9.1.3 in  \cite{LR}
 all the eigenvalues
of the matrix $4\tilde\Gamma\tilde Q -\tilde B$
have positive real parts, therefore (\ref{Lyap1}) has the unique solution $\tilde D$ given by
$$
\tilde D= 2\int_{-\infty}^0 e^{(4\tilde \Gamma \tilde Q - \tilde B)^* t} e^{(4\tilde\Gamma\tilde Q -
\tilde B)t}{\rm d}t,
$$
which is a symmetric positive definite matrix. Now introduce
functions
\begin{equation}
\label{testf}
W^{\pm}_{\delta}(z^\prime,z_N)=(\tilde \Gamma \pm\delta
\tilde D)_{ij}z_i^\prime z_j^\prime \pm\delta z^2_N
\end{equation}
depending on the parameter $\delta >0$.

\begin{lem}
\label{Lcompar}
Let  $\delta>0$ be sufficiently small. Then, for small $|z|\not=0$ such that  $X(z)\in\overline\Omega$,
we have 
\begin{equation}
\label{subsupersolut}
W^{-}_{\delta}(z)< W(X(z))< W^{+}_{\delta}(z).
\end{equation}
\end{lem}
\begin{proof}
By virtue of the definition of $W^{\pm}_\delta$ it suffices to  prove (\ref{subsupersolut}) with non strict
inequalities in place of strict ones an then pass to slightly bigger $\delta$.

The proof of the inequality $W_\delta^{-}\leq W$ is based on the
following two facts. First, we use  the fact that
$W(x)=d_H(x,\xi)$ in a neighborhood of $\xi$. Moreover, for a
given $\delta^\prime>0$ there exists $\delta>0$ such that if
$|x-\xi|<\delta$ then the minimization in (\ref{dfunction}) is
actually restricted to testing curves $\eta(\tau)$ which do not
leave the set $\{|\eta-\xi|<\delta^\prime\}$ (otherwise arguing as
in \cite[Lemma 19]{PR}  one can show that $\xi$ is not an isolated
point of the Aubry set ${\mathcal A}_H$, contradicting
(\ref{H1})). Second, considering, with a little abuse of notation,
$W_\delta^{-}(x)=W_\delta^{-}(X^{-1}(x))$ we have for sufficiently
small $\delta>0$
\begin{equation}
\label{testsubsol}
 H(\nabla W_\delta^{-},x)\leq -\delta |x-\xi|^2 \ \text{in}\ \Omega,
 \ \text{and}\ \frac{\partial W_\delta^{-}}{\partial\nu}=0\
 \text{on}\ \partial \Omega,
\end{equation}
when $|x-\xi|<\delta^\prime$ with some $\delta^\prime>0$
independent of $\delta$. This follows from the construction
(\ref{testf}) of $W_\delta^{\pm}$ and (\ref{Ric1}), (\ref{Lyap1}),
also taking into account (\ref{H3newcoord}).

Assume that $|x-\xi|<\delta$, and let $\eta(\tau)$ be a solution of
(\ref{Skor}) satisfying $\eta(0)=x$, $\eta(t)=\xi$ with a control
$v(\tau)$ such that $|\eta(\tau)-\xi|<\delta^\prime$ for all
$0\leq \tau\leq t$. Then
$$
W_\delta^-(x)=-\int_0^t \nabla W^-_\delta(\eta)\cdot\dot\eta\,
{\rm d}\tau= \int_0^t \nabla W^-_\delta(\eta)\cdot(-v(\tau))\,
{\rm d}\tau,
$$
where we have used the fact that $\frac{\partial
W^-_\delta}{\partial \nu}=0$ on $\partial \Omega$. It follows  by Fenchel's inequality
$p\cdot(-v)\leq L(-v,\eta)+H(p,\eta)$ that
$$
W_\delta^-(x)\leq \int_0^t L(-v,\eta)\,{\rm d}\tau +\int_0^t
H(\nabla W^-_\delta,\eta){\rm d}\tau \leq \int_0^t
L(-v,\eta)\,{\rm d}\tau.
$$
Therefore by (\ref{dfunction}) we obtain  $W_\delta^-(x)\leq
W(x)$.


In order to prove the second inequality in  (\ref{subsupersolut})
for a given $x=X(z^\prime,z_N)$ we construct a test curve
$\eta(\tau)$ first on a small interval $(0,\Delta t)$ by setting
$\eta(\tau)=X(z^\prime,\zeta_N(\tau))$, $\zeta_N(\tau)$ being the
solution of ODE
$\dot\zeta_N(\tau)=b_i(X(z^\prime,\zeta_N))\mathcal{T}^{-1}_{Ni}(z^\prime,
\zeta_N)$ with the initial condition $\zeta_N(0)=z_N$, and choosing
$\Delta t$ from the conditions $\zeta_N(\Delta t)=0$,
$\zeta_N(\tau)>0$ for $\tau<\Delta t$. Thanks to (\ref{H3newcoord})
we have $\Delta t=O(z_N)$. Then, since
$$
\begin{aligned}
\dot\eta_i
&
=\frac{\partial X_i}{\partial
z_N}(z^\prime,\zeta_N)b_k(\eta) \mathcal{T}^{-1}_{Nk}(z^\prime,\zeta_N)\\
&
=\frac{\partial X_i}{\partial
z_j}(z^\prime,\zeta_N)b_k(\eta) \mathcal{T}^{-1}_{jk}(z^\prime,\zeta_N)
-\frac{\partial X_i}{\partial z^\prime_j}(z^\prime,\zeta_N)b_k(\eta)
\mathcal{T}^{-1}_{jk}(z^\prime,\zeta_N)=b_i(\eta)+O(|z|)
\end{aligned}
$$
(recall that the tangential component $b_\tau$ on $\partial \Omega$ vanishes at the point $\xi$)
we obtain
\begin{equation}
\label{firstint}
 \int_0^{\Delta t} L(-\dot\eta,\eta){\rm
d}\tau=O(|z|^3).
\end{equation}
Next we construct $\eta(\tau)$ for $\tau>\Delta t$ which connects
point $X(z^\prime,0)$ to $\xi$. Following closely the line of
Lemma \ref{Lsupers} we introduce $\zeta^\prime(\tau)$ by solving
the equation $\dot\zeta^\prime(\tau)=-(4 \tilde Q\tilde \Gamma - \tilde B)\zeta^\prime$
with the initial condition $\zeta^\prime(\Delta t)=z^\prime$ and
set $\eta(\tau)=X(\zeta^\prime(\tau),0)$. Then
$\eta(\tau)$ solves (\ref{Skor}) for $\tau>\Delta t$ with
$$
v(\tau):=\dot\eta(\tau)+\nu(\eta)b_\nu(\eta)-
4\frac{\partial X}{\partial z_N}(\zeta^\prime,0)\mathcal{T}^{-1}_{Ni}(0)a_{ij}(\xi)\mathcal{T}^{-1}_{lj}(0)
\tilde \Gamma_{lm}\zeta^\prime_m
$$
(note that $\frac{\partial X}{\partial
z_N}(\zeta^\prime,0)=-\nu(\eta)$ and
$b_\nu(\eta)+4\mathcal{T}^{-1}_{Ni}(0)a_{ij}(\xi)\mathcal{T}^{-1}_{lj}(0)\Gamma_{lm}\zeta^\prime_m>0$
as soon as $|z^\prime|$ is sufficiently small) and using
(\ref{Ric1}) we obtain
\begin{equation}
\label{secondint}
\begin{aligned}
 \int_{\Delta t}^\infty L(&-v(\tau),\eta(\tau))\,{\rm
d}\tau=\frac{1}{4}\int_{\Delta t}^\infty
a^{ij}(\xi)(-v_i+b_i)(-v_j+b_j)\,{\rm d}\tau+O(|z|^3)
\\
&=4\int_{\Delta t}^\infty a^{ij}(\xi)\Bigl(\frac{\partial X_i}
{\partial z_k^\prime}(0) \tilde Q_{kl}\tilde \Gamma_{lm}\zeta^\prime_m+\frac{\partial X_i}
{\partial z_N}(0)\mathcal{T}^{-1}_{Nk}(0)a_{kl}(\xi)\mathcal{T}^{-1}_{ml}(0)\tilde\Gamma_{m n}\zeta^\prime_n\Bigr)\\
&\quad\quad\quad
\quad\quad\times \Bigl(\frac{\partial X_j}
{\partial z_k^\prime}(0) \tilde Q_{kl}\tilde \Gamma_{lm}\zeta^\prime_m+\frac{\partial X_j}
{\partial z_N}(0)\mathcal{T}^{-1}_{Nk}(0)a_{kl}(\xi)\mathcal{T}^{-1}_{ml}(0)\tilde \Gamma_{m n}\zeta^\prime_n\Bigr)
\,{\rm
d}\tau+O(|z|^3)\\
&=4\int_{\Delta t}^\infty\tilde \Gamma\zeta^\prime\cdot
\tilde Q\tilde \Gamma\zeta^\prime\,{\rm d}\tau+O(|z|^3)
\\
&=-2\int_{\Delta
t}^\infty\tilde \Gamma\zeta^\prime\cdot \dot\zeta^\prime\,{\rm d}\tau +
\int_{\Delta t}^\infty\tilde \Gamma\zeta^\prime\cdot
\bigl(\dot\zeta^\prime+\tilde B\zeta^\prime)\,{\rm d}\tau+O(|z|^3)
=\tilde \Gamma_{ij} z^\prime_i z^\prime_j+O(|z|^3).
\end{aligned}
\end{equation}
The required upper bound $W\leq W_\delta^{+}$ now follows from
(\ref{firstint}) and (\ref{secondint}).
\end{proof}

Thus functions $W^{\pm}_\delta$ satisfy conditions of Lemma \ref{MainTechLem}, moreover
it follows from (\ref{testf}) in conjunction with (\ref{Ric1}), (\ref{Lyap1}), taking also into account (\ref{H3newcoord}),
that
\begin{equation*}
S(\nabla_z W^{+}_\delta (z),z)\geq 0\ \text{and}\  S(\nabla_z W^{-}_\delta (z),z)\leq 0\ \text{when} \ |z|
\ \text{is sufficiently small}.
\end{equation*}
Then we set  $$
W^{\pm}_{\delta,\ve}(z^\prime,z_N)
=W^{\pm}_\delta(z^\prime,z_N)\mp\ve^2z_N,
$$
and verify (similarly to the case of interior fixed points) that
conditions  (\ref{forLowerBound}) and (\ref{forUpperBound})
are satisfied.
%
%
%

\section{Construction of test functions: case of limit cycles in $\Omega$}
\label{LCin}
We proceed with the case when  a significant component
of the Aubry set  $\mathcal{A}_H$ is a limit cycle, assuming first
that it is situated entirely inside $\Omega$. Namely, let $\xi(t)$
be a periodic solution of the ODE $\dot\xi=b(\xi)$ whose minimal
period is $P>0$. We assume that $\mathcal{C}=\{\xi(t): t\in[0,P)
\}\subset \Omega$, $b(x)\not=0$ on $\mathcal{C}$ and $\mathcal{C}$
is a hyperbolic limit cycle, i.e. the linearized Poincar\'{e} map
associated to this cycle has no eigenvalues on the unit circle. In
order to study the local behavior  of
$W$ near the
cycle $\mathcal{C}$, perform a $C^2$-smooth change of coordinates
$x=X(z_1,\dots,z_{N-1},z_N)$ with $z_N$ representing the arc
length along the cycle and $z'=(z_1,\dots,z_{N-1})$ being some fixed
Cartesian coordinates in the hyperplanes orthogonal to the cycle.
Also we assume that $\mathcal{C}$ is oriented by the tangent
vector $b(\xi)/|b(\xi)|$, and $z'=0$ on $\mathcal{C}$. With this
change of coordinates equations (\ref{HamEq}) and (\ref{PrII}) take the form similar to
(\ref{nHamJac}) and (\ref{nperturbed}).
Assuming as above that $W(\mathcal{C})=0$, we postulate in the vicinity of the cycle (for sufficiently small $|z'|$)  the following ansatz for $W$:
\begin{equation}
\label{anscycle}
W(X(z',z_N))=\overline{\Gamma}_{ij}(t) z'_i z'_j+o(|z'|^2),
\end{equation}
where $t$ refers to the parametrization of the cycle determined by the equation $\dot\xi=b(\xi)$, $t\in[0,P)$.
Substitute $W$ in (\ref{nHamJac})
to find, after collecting quadratic terms and neglecting higher order terms,
\begin{equation}
\label{dRic}
\dot{\overline{\Gamma}}=
4\overline{\Gamma}\, \overline{Q}\, \overline{\Gamma} -\overline{\Gamma} \,\overline{B}-
\overline{B}^*\overline{\Gamma},
\end{equation}
where $\overline{Q}(t)$ and $\overline{B}(t)$ are $P$-periodic
$(N-1)\times(N-1)$ matrices whose entries are given by
\begin{equation}
\label{mQ} \overline{Q}_{kl}(t)= a_{ij}(\xi(t))\,
\mathcal{T}^{-1}_{ki}\bigl(0,z_N(\xi(t))\bigr)\,\mathcal{T}^{-1}_{lj}\bigl(0,z_N(\xi(t))\bigr),
\end{equation}
\begin{equation}
\label{mB} \overline{B}_{kl}(t)=
\mathcal{T}^{-1}_{ki}\bigl(0,z_N(\xi(t))\bigr)\frac{\partial b_i} {\partial
x_j}\bigl(\xi(t)\bigr)   \frac{\partial X_j} {\partial
z_l}\bigl(0,z_N(\xi(t))\bigr) -
\mathcal{T}^{-1}_{kj}\bigl(0,z_N(\xi(t))\bigr) \frac{\rm d}{{\rm d}t} \biggl(
\frac{\partial X_j} {\partial
z_l}\bigl(0,z_N(\xi(t))\bigr)\biggr).
\end{equation}
{Recall that  $\mathcal{T}^{-1}_{ij}\bigl(z^\prime,z_N)$ denote the entries of the matrix inverse to
$\bigl(\frac{\partial X_i}{\partial z_j}(z^\prime,z_N)\bigr)_{i,j=\overline{1,N}}$ and for brevity abusing slightly
the notation we set
\begin{equation}\label{def_Tkras}
\mathcal{T}^{-1}_{ij}(t):=\mathcal{T}^{-1}_{ij}(0,z_N(\xi(t)),\quad {\mathcal T}_{ij}(t)=
\frac{\partial X_i}{\partial z_j}\bigl(0,z_N(\xi(t))\bigr).
\end{equation}}

The matrix $\overline{Q}(t)$ being positive definite, it is known \cite{AFIJ} that Riccati equation   (\ref{dRic})
 has a maximal symmetric P-periodic solution $\overline{\Gamma}(t)$. We next show that (\ref{anscycle}) does
hold {with the mentioned  maximal solution $\overline{\Gamma}(t)$} under our standing hyperbolicity assumption on $\mathcal{C}$. Note that the ODE $\dot z'=\overline{B}z'$
 corresponds to the linearization of $\dot x=b(x)$ on the cycle $\mathcal{C}$  written in local coordinates;
 thus assuming the hyperbolicity of
$\mathcal{C}$ we require  that the fundamental  solution of the
ODE
$$
\frac{\partial\overline{\Phi}}{\partial t}(t,\tau)=\overline{B}(t)\overline{\Phi}(t,\tau), \qquad   \overline{\Phi}(\tau,\tau)=I,
$$
evaluated at $t=\tau+P$, has no eigenvalues with absolute value equal to 1.

\begin{lem}
\label{Lemupperb}
The following bound holds uniformly in $t\in[0,P) $ for sufficiently small $|z'|$,
\begin{equation}\label{est_in_l8}
W(X(z',z_N(\xi(t)))\leq\overline{\Gamma}_{ij}(t)z'_i z'_j
+C|z'|^3\log\frac{1}{|z'|}.
\end{equation}
\end{lem}
\begin{proof}
We make use of variational representation  (\ref{dfunction}).
{A natural guess about  optimal test curve in (\ref{dfunction}) is that its first $N-1$
local coordinates are  given by (cf. Sections \ref{FixP} and \ref{FP})  }
\begin{equation}
\label{auxODE}
\dot\zeta'(\tau)=(\overline{B}-4\overline{Q}\,\overline{\Gamma})\zeta'(\tau),\quad\tau>t,\quad\zeta'(t)=z'.
\end{equation}
Thanks to Theorem 5.4.15 in \cite{AFIJ} the solutions of (\ref{auxODE}) are exponentially
stable, i. e.  $|\zeta'(\tau)|\leq Ce^{-\delta\tau}|z'|$ for some
$\delta>0$. The choice of the last local coordinate is a little bit involved. We set $\eta(\tau):=X(\zeta^\prime(\tau),z_N(\xi(\tau))+\zeta_N(\tau))$ and want to choose $\zeta_N(\tau)$
in such a way that $|\zeta_N(\tau)|<C|z^\prime|$, and
\begin{equation}\label{prop_eta}
b_i(\eta(\tau))-\dot\eta_i(\tau)=4a_{ij}(\xi(\tau))\mathcal{T}^{-1}_{lj}\bigl(0,\xi(\tau)\bigr)\overline{\Gamma}_{lm}(\tau)
\zeta'_m(\tau)+O(|z'|^2).
\end{equation}
We skip  for a moment the proof of the existence of such $\zeta_N$. It will be given later on.

Considering (\ref{prop_eta}) we obtain
$$
\begin{aligned}
\int\limits_t^TL(-\dot\eta,\eta)\mathrm{d}\tau &=
\frac{1}{4}\int\limits_t^T  a^{ij}(\eta(\tau))(-\dot\eta_i(\tau)+b_i(\eta(\tau))) (-\dot\eta_j(\tau)+b_j(\eta(\tau))){\rm d}\tau \\
&\leq 4\int\limits_t^T (\overline{\Gamma}(\tau)\zeta'(\tau))\cdot (\overline{Q}(\tau)\overline{\Gamma}(\tau)
\zeta'(\tau)){\rm d}\tau+CT|z'|^3
\end{aligned}
$$
In view of (\ref{auxODE}) and (\ref{dRic}) we have
$$
\begin{aligned}
\int\limits_t^TL(-\dot\eta,\eta)\mathrm{d}\tau &\leq\int\limits_t^T (\overline{\Gamma}\zeta')\cdot (\overline{B}\zeta'-\dot\zeta'){\rm d}\tau+CT|z'|^3\\
&=-\int\limits_t^T \frac{\rm d}{{\rm d}\tau}(\overline{\Gamma}\zeta'\cdot\zeta'){\rm d}\tau
+\int\limits_t^T (\overline{\Gamma}\zeta'\cdot(\overline{B}\zeta'+\dot\zeta')+\dot{\overline{\Gamma}}\zeta'\cdot\zeta'){\rm d}\tau
+CT|z'|^3\\
&\leq\overline{\Gamma} (t)z'\cdot z' +\int\limits_t^T (\dot{\overline{\Gamma}}-4\overline{\Gamma} \, \overline{Q}\, \overline{\Gamma}
+\overline{\Gamma}\, \overline{B}+\overline{B}^*\,\overline{\Gamma})\zeta'\cdot\zeta'){\rm d}\tau +CT|z'|^3\\
&=\overline{\Gamma}(t)z'\cdot z'+CT|z'|^3.
\end{aligned}
$$
If we choose $T:=\frac1\delta\log\frac1{|z'|}$, then $\mathrm{dist}(\eta(T), \mathcal{C})=O(|z'|^2)$, and
\begin{equation}\label{app_lagr}
\int\limits_t^TL(-\dot\eta,\eta)\mathrm{d}\tau\leq\overline{\Gamma}(t)z'\cdot z'+C|z'|^3\log\frac1{|z'|}.
\end{equation}

For constructing $\zeta_N$  we will need the following facts.
From the definition of $\mathcal{T}_{ij}$ in \eqref{def_Tkras} it follows that
$\mathcal{T}_{iN}(\tau)=\dot\xi_i(\tau)/|\dot\xi(\tau)|$. Then, since
$\ddot \xi_i (\tau)=\frac{\partial b_i}{\partial x_j}(\xi(\tau)) \dot\xi_j(\tau)$, we have
\begin{equation*}
\label{invar}
\frac{\partial b_i}{\partial x_j}(\xi(\tau))\mathcal{T}_{jN}(\tau)- \mathcal{\dot T}_{iN}(\tau)=
\frac{\dot\xi_i(\tau)\dot\xi_j(\tau)\ddot\xi_j(\tau)}{|\dot\xi(\tau)|^3}
=\mathcal{T}_{iN}(\tau)\frac{\dot\xi_j(\tau)\ddot\xi_j(\tau)}{|\dot\xi(\tau)|^2}
,\quad \forall i=1,\dots,N.
\end{equation*}
Multiplying this by  $
\mathcal{T}^{-1}(\tau)$,
we conclude that
\begin{equation}
\label{invar1}
\mathcal{T}^{-1}_{ki}(\tau)\Bigl(\frac{\partial
b_i}{\partial x_j}(\xi(\tau))\mathcal{T}_{jN}(\tau) -
\mathcal{\dot T}_{iN}(\tau)\Bigr)=
\begin{cases}
 \dot\xi_j(\tau)\ddot\xi_j(\tau)/|\dot\xi(\tau)|^2,\quad {\rm if}\ k=N,\\
0\quad {\rm if}\ k<N.
\end{cases}
\end{equation}
We proceed with constructing $\zeta_N$. By the definition of $\eta$ we have
\begin{equation}
\begin{aligned}
\dot\eta_i(\tau)&-b_i(\eta(\tau))=\frac{\rm d}{{\rm
d}t}\Bigl(\xi_i(\tau)+\mathcal{
T}_{ik}(\tau)\zeta^\prime_k(\tau)+\mathcal{ T}_{iN}(\tau)
\zeta_N(\tau)\Bigr)-b_i(\eta(\tau))+O(|z'|^2)
\\
&=\Bigl( \mathcal{\dot T}_{ik}(\tau)-\frac{\partial b_i}{\partial
x_j}(\xi(\tau))\mathcal{T}_{jk}(\tau)\Bigr)\zeta'_k(\tau)
+\Bigl( \mathcal{\dot T}_{iN}(\tau)-\frac{\partial b_i}{\partial x_j}(\xi(\tau))\mathcal{T}_{jN}(\tau)\Bigr)\zeta_N(\tau)\\
&\ \ +\mathcal{T}_{ik}(\tau)\dot\zeta'_k(\tau)+\mathcal{
T}_{iN}(\tau)\dot\zeta_N(\tau)+O(|z'|^2)\\
&=\mathcal{T}_{ik}(\tau)\mathcal{T}^{-1}_{kr}(\tau)\Big\{
\Bigl( \mathcal{\dot T}_{rk}(\tau)-\frac{\partial b_r}{\partial
x_j}(\xi(\tau))\mathcal{T}_{jk}(\tau)\Bigr)\zeta'_k(\tau)
+\Bigl( \mathcal{\dot T}_{rN}(\tau)-\frac{\partial b_r}{\partial x_j}(\xi(\tau))\mathcal{T}_{jN}(\tau)\Bigr)\zeta_N(\tau)\Big\}\\
&\ \ +\mathcal{T}_{ik}(\tau)\dot\zeta'_k(\tau)+\mathcal{
T}_{iN}(\tau)\dot\zeta_N(\tau)+O(|z'|^2).
\end{aligned}
\label{preobrazov}
\end{equation}
Substituting for $\dot\zeta'$ the expression on the right-hand side of \eqref{auxODE} and considering
(\ref{invar1}) yields
$$
b_i(\eta(\tau))-\dot\eta_i(\tau)=4a_{ij}(\xi(\tau))\mathcal{T}^{-1}_{lj}\bigl(0,\xi(\tau)\bigr)\overline{\Gamma}_{lm}(\tau)
\zeta'_m(\tau)+O(|z'|^2)-\mathcal{T}_{iN}(\tau)\dot\zeta_N(\tau)+\mathcal{T}_{iN}(\tau)R(\tau),
$$
where
$$
\begin{aligned}
R(\tau)&=\mathcal{T}^{-1}_{Ni}(\tau)\frac{\partial
b_i}{\partial x_j}(\xi(\tau))
\biggl( \mathcal{T}_{jN} (\tau) \zeta_N(\tau)+\mathcal{T}_{jl}(\tau) \zeta'_l(\tau)\biggr) \\
&-\mathcal{T}^{-1}_{Ni}(\tau) \biggl( \mathcal{\dot T}_{iN}(\tau) \zeta_N(\tau)+ \mathcal{\dot T}_{il}(\tau) \zeta'_l(\tau)\biggr) 
-4 \mathcal{T}^{-1}_{Ni}(\tau) a_{ij}(\xi(\tau))
\mathcal{T}^{-1}_{lj}(\tau)\overline{\Gamma}_{lm}(\tau)
\zeta'_m(\tau)
\end{aligned}
$$
Thus, in order to make (\ref{prop_eta})  hold, we choose $\zeta_N(\tau)$ as a solution of the following equation
\begin{equation}
\label{auxODEN}
\dot\zeta_N(\tau)=R(\tau)
\end{equation}
with the initial condition $\zeta_N(t)=0$.

From (\ref{invar1}) it follows that $|\dot\xi(\tau)|$ solves
$$
\frac{\rm d}{{\rm d}t}|\dot\xi(\tau)|= \bigl(
\mathcal{T}^{-1}_{Ni}(\tau)\frac{\partial
b_i}{\partial x_j}(\xi(\tau))\mathcal{T}_{jN}(\tau)-
\mathcal{T}^{-1}_{Ni}(\tau)\mathcal{\dot T}_{iN}(\tau)
\bigr)|\dot\xi(\tau)|,
$$
and we can write the solution $\zeta_N(\tau)$ of (\ref{auxODEN}) as
\begin{equation}
\begin{aligned}
\zeta_N(\tau)&=|\dot\xi(\tau)|\int\limits_0^{\tau}
\biggl(\mathcal{T}^{-1}_{Ni}(s)\frac{\partial
b_i}{\partial x_j}(\xi(s))
 \mathcal{T}_{jl} (s) \zeta'_l(s)
-\mathcal{T}^{-1}_{Ni}(s) \mathcal{\dot T}_{il}(s) \zeta'_l(s)\biggr) \frac{{\rm d}s}{|\dot\xi(s)|}\\
&-4|\dot\xi(\tau)|\int\limits_0^{\tau} \mathcal{T}^{-1}_{Ni}(s) a_{ij}(\xi(s))
\mathcal{T}^{-1}_{lj}(s)\overline{\Gamma}_{lm}(s)
\zeta'_m(s) \frac{{\rm d}s}{|\dot\xi(s)|}.
\end{aligned}
\label{integral}
\end{equation}
From (\ref{integral}) we derive the uniform bound $|\zeta_N(\tau)|\leq C|z'|$.

It remains to construct $\eta(\cdot)$ for  $\tau>T$ in such a way that it reaches the cycle
in a finite time. To this end we set
\begin{equation}
\label{fintest}
\eta(\tau)=
X\big(\zeta'(T)(T+1-\tau),z_N(\xi(\tau))+\zeta_N(T)|\dot\xi(\tau)|/|\dot\xi(T)|\big),\quad T\leq\tau\leq T+1.
\end{equation}
Then for every $i=1,\dots,N$
$$
\begin{aligned}
\dot\eta_i(\tau)&=\frac{\rm d}{{\rm d} \tau}
\bigl(\xi_i(\tau)+
{\mathcal T}_{iN}(\tau)\zeta_N(T)|\dot\xi(\tau)|/|\dot\xi(T)|\bigr)+O(|z^\prime|^2)
\\&=\dot\xi_i(\tau)+\ddot \xi_i(\tau)\zeta_N(T)/|\dot\xi(T)|+
O(|z^\prime|^2)
\\
&=b_i(\xi(\tau))+\frac{\partial b_i}{\partial
x_j}(\xi(\tau))\dot\xi_j(\tau)
\zeta_N(T)/|\dot\xi(T)|+O(|z^\prime|^2)=b_i(\eta(\tau))+O(|z^\prime|^2).
\end{aligned}
$$
 Therefore, the following bound holds
\begin{equation}
\label{hvost}
\int\limits_T^{T+1} a^{ij}(\eta(\tau))(-\dot\eta_i+b_i(\eta)) (-\dot\eta_j+b_j(\eta)){\rm d}\tau\leq C|z'|^4.
\end{equation}
Combining the last relation with (\ref{app_lagr}) yields (\ref{est_in_l8}).
\end{proof}

In order to construct a sub- and supersolutions of  (\ref{nHamJac}) we consider the
solution $\overline{D}(t)$ of the matrix equation
\begin{equation}
\label{dLyap}
\dot{\overline{D}} +\overline{D}(\overline{B}-4\overline{Q}\,\overline{\Gamma})
+(\overline{B}-4\overline{Q}\,\overline{\Gamma})^*\overline{D}
=-2I,
\end{equation}
given by
\begin{equation}
\label{defD}
\overline{D}(t)=2\int\limits_t^\infty\overline{ \Psi}^*(\tau,t) \overline{\Psi}(\tau,t){\rm d}\tau,
\end{equation}
where $\overline{\Psi}(\tau,t)$ is the fundamental matrix solution of
$$
\frac{\partial\overline{\Psi}}{\partial \tau}
=(\overline{B}(\tau)-4\overline{Q}(\tau)\overline{\Gamma}(\tau))\overline{\Psi},\quad
\overline{\Psi}(t,t)=I.
$$
As already mentioned in the proof of Lemma \ref{Lemupperb}
this solution $\overline{\Psi}(\tau,t)$ decays exponentially as $\tau\to +\infty$
and therefore the integral in  (\ref{defD}) converges. Then it defines a $P$-periodic positive symmetric
solution of (\ref{dLyap}). It follows from
(\ref{dRic}) and (\ref{dLyap}) that $\overline{\Gamma}^{\pm}_{\delta}:=\overline\Gamma\pm\delta \overline{D}$ satisfy
for sufficiently small $\delta>0$
$$
4 \overline{\Gamma}^+_{\delta}\,\overline{Q}\,\overline{\Gamma}^+_{\delta}-\frac{\rm d}{{\rm d}t}\overline{\Gamma}^+_{\delta}
-\overline{\Gamma}^+_{\delta}\,\overline{B}-\overline{B}^*\overline{\Gamma}^+_{\delta}\geq \delta I
\quad\text{and}\quad
4 \overline{\Gamma}^-_{\delta}\,\overline{Q}\,\overline{\Gamma}^-_{\delta}
-\frac{\rm d}{{\rm d}t}\overline{\Gamma}^-_{\delta}-\overline{\Gamma}^+_{\delta}\,\overline{B}-\overline{B}^*
\overline{\Gamma}^-_{\delta}\leq -\delta I.
$$
Now define the functions $W^{\pm}_{\delta}(z)$ by
$$
W^{\pm}_{\delta}(z)= (\overline{\Gamma}^\pm_{\delta}(t))_{ij}z'_i
z'_j\quad {\rm where} \  z_N=z_N(\xi(t)),
$$
these functions satisfy
\begin{equation}
\label{latterproperty} S(\nabla_z W^-_{\delta}(z),z)\leq
-\frac{\delta}{2}|z'|^2\quad \text{and}\quad S(\nabla_z
W^+_{\delta}(z),z)\geq \frac{\delta}{2}|z'|^2  \quad {\rm for\
sufficiently\ small\ } |z'|.
\end{equation}
The latter inequalities follow directly form the definitions of
$W^{-}_{\delta}(z)$ and $W^{+}_{\delta}(z)$.


 \begin{lem}
\label{Lemlowerb}
For sufficiently small $\delta>0$
the strict pointwise inequalities  $W_\delta^-(z)< W(X(z))<W_\delta^+(z)$ hold
for $z'$ from a punctured neighborhood of zero.
\end{lem}

\begin{proof} The first inequality $W_\delta^-(z)< W(X(z))$ can be proved similarly to Lemma 19 in \cite{PR}
(see also  the proof of Lemma \ref{Lcompar}), using (\ref{latterproperty}). The second inequality
$W(X(z))<W_\delta^+(z)$ follows immediately from Lemma \ref{Lemupperb}.
\end{proof}

At this point we have constructed functions $W_\delta^{\pm}$ satisfying conditions of Lemma \ref{MainTechLem}.
Next  we
define the test functions $W^{\pm}_{\delta,\ve}$ by
$$
W^{\pm}_{\delta,\ve}:=W_\delta^{\pm}-\ve\overline{\Phi}^{\pm}_\delta(t),
$$
where $z_N$ and $t$ are related by $z_N=z_N(\xi(t))$, and
$\overline{\Phi}^{\pm}_\delta(t)$ are periodic solutions  of the
ODEs
\begin{equation}
\label{phimin} \frac{\rm d}{{\rm d}
t}\overline{\Phi}^{\pm}_\delta(t)=-2{\rm
tr}(\overline{Q}(t)\overline{\Gamma}^{\pm}_\delta(t))+c(\xi(t))+\frac{2}{P}\int\limits_0^P{\rm
tr}(\overline{Q}(\tau)\overline{\Gamma}^{\pm}_\delta(\tau)){\rm
d}\tau- \frac{1}{P}\int\limits_0^Pc(\xi(\tau)){\rm d}\tau.
\end{equation}
The first two terms on the right-hand side here are introduced in order to compensate the discrepancy
of order $\ve$ in equation (\ref{nperturbed}). Indeed,
the test functions $W^{\pm}_{\delta,\ve}$ constructed in this way
satisfy for sufficiently small $|z'|$
\begin{equation*}
\begin{aligned}
{\pm}a_{ij} \left(X(z)\right) \alpha_{ki}
(z)&\frac{\partial}{\partial z_k}\biggl(
\alpha_{lj}(z)\frac{\partial W^{\pm}_{\delta,\ve}}{\partial
z_l}\biggr)
\mp
\frac{1}{\ve}S(\nabla_z W^\pm_{\delta,\ve},z)\mp c(X(z))\\
&\leq \frac{1}{P}\int\limits_0^P c(\xi(\tau)){\rm
d}\tau-\frac{2}{P}\int\limits_0^P{\rm
tr}(\overline{Q}(\tau)\overline{\Gamma}^{\pm}_\delta(\tau)){\rm
d}\tau +O(\ve+|z^\prime|).
\end{aligned}
\end{equation*}
In order to complete the proof
of the fact that $W^{\pm}_{\delta,\ve}$ satisfy
(\ref{forLowerBound}) and (\ref{forUpperBound}) it remains to
observe that $\int_0^P {\rm tr}
(\overline{Q}(\tau)\overline{\Gamma}^{\pm}_\delta(\tau)){\rm
d}\tau\to \int_0^P {\rm tr}
(\overline{Q}(\tau)\overline{\Gamma}(\tau)){\rm d}\tau$ as
$\delta\to 0$ and use the identity
$$
2\int_0^P {\rm tr}
(\overline{Q}(\tau)\overline{\Gamma}(\tau)){\rm d}\tau=\sum\limits_{\Theta_i>1}
\log\Theta_i
$$
(see Proposition 5.1 in \cite{PRR}), where $\Theta_i$ are
absolute values of eigenvalues of the linearized  Poincar\'{e} map
(corresponding to the ODE $\dot x=b(x)$ near $\mathcal{C}$).


\section{Construction of test functions: case of limit cycles on $\partial \Omega$}
\label{LimConboundary} In the case when ODE $\dot x=b_\tau(x)$ on
$\partial \Omega$ has a limit cycle  $\mathcal{C}$ which is
significant component of the Aubry set, the analysis combines the
ideas of Section \ref{FP} and Section \ref{LCin}. We pass to the
local coordinates in a neighborhood of $\mathcal{C}$ via a map
$x=X(z_1, \dots, z_{N-1}, z_N)$, where $z_N=z_N(x)$  is the
distance from $x$ to $\partial\Omega$ (positive for $x\in\Omega$)
and $(z_1, \dots, z_{N-1})$ are coordinates on $\partial\Omega$. The coordinate
$z_{N-1}(x)$ represents the arc length parametrization
on $\mathcal{C}$ and other
coordinates $z^\prime=(z_1, \dots, z_{N-2})$ are chosen so that
the map
$X(z^\prime,z_{N-1}, z_N)$ is $C^2$-smooth, moreover $z'=0$ when
$x\in \mathcal{C}$,
 and
 $\bigl( \frac{\partial X_i}{\partial z_j}(z)\bigr)_{i,j=\overline{1,N}}$
is an orthogonal matrix when $z_N=0$ and $z'=0$ (on the cycle).
 This change of coordinates leads to equations of the
form  (\ref{nHamJac}) and
 (\ref{nperturbed}) for $W(X(z))$ and $W_\ve(X(z))$.

We use the following ansatz for $W$,
\begin{equation}
\label{ansbcy} W(X(z))=\widehat\Gamma_{ij}(t)z'_i z'_j +o(|z'|^2),
\end{equation}
where $\widehat\Gamma$ is now $(N-2)\times (N-2)$ symmetric
$P$-periodic matrix ($P$ being the period of the cycle $\mathcal
C$), and $t$ refers to the parametrization $t\rightarrow \xi(t)$
of $\mathcal{C}$ such that $\dot\xi(t)=b_\tau(\xi(t))$. Moreover,
$\widehat{\Gamma}$ is chosen to be the maximal $P$-periodic solution
of the Riccati matrix equation
\begin{equation*}
\frac{\rm d}{{\rm d} t}\widehat{\Gamma}= 4\widehat{\Gamma} \widehat{Q}
\widehat{\Gamma} -\widehat{\Gamma} \widehat{B}- \widehat{B}^*\widehat{\Gamma},
\end{equation*}
with $(N-2)\times (N-2)$ matrices $\widehat{Q}(t)$ and $\widehat{B}(t)$
whose entries are given by the same formulas as (\ref{mQ}) and
(\ref{mB}).


\begin{lem}
For sufficiently small $|z'|$ and $|z_N|$ the following bound holds uniformly in $t\in[0,P)$
\begin{equation}
W(X(z',z_{N-1}(\xi(t)),z_N)\leq\widehat\Gamma_{ij}(t)z'_i z'_j
+C(|z'|^2\log\frac{1}{|z'|}+|z_N||z'|^2+|z_N|^3). \label{OTSENKA}
\end{equation}
\end{lem}
\begin{proof} First consider the case $z_N=0$. As in Lemma
\ref{Lemupperb} we use representation (\ref{dfunction}) and
consider the solution $\zeta^\prime(\tau)$ of the ODE
$\dot\zeta^\prime(\tau)=(\widehat{B}-4\widehat{Q}\widehat{\Gamma})\zeta(\tau)$
for $\tau>t$ with the initial condition
$\zeta^\prime(t)=z^\prime$. It decays exponentially as
$\tau\to\infty$, $|\zeta^\prime|\leq C e^{-\delta \tau}
|z^\prime|$ for some $\delta>0$. Next we introduce $\zeta_{N-1}$
analogously to $\zeta_{N}$ introduced in Lemma \ref{Lemupperb}, i.e.
$\zeta_{N-1}$ solves
\begin{equation*}
\begin{aligned}
\dot\zeta_{N-1}(\tau)&=\mathcal{T}^{-1}_{(N-1)i}\bigl(0,z_{N-1}(\xi(\tau),0\bigr)\frac{\partial
b_i}{\partial x_j}(\xi(\tau))
\biggl( \mathcal{T}_{j(N-1)} (\tau) \zeta_{N-1}(\tau)+\mathcal{T}_{jl}(\tau) \zeta'_l(\tau)\biggr) \\
&-\mathcal{T}^{-1}_{(N-1)i}\bigl(0,z_{N-1}(\xi(\tau)),0\bigr) \biggl( \mathcal{\dot T}_{i(N-1)}(\tau) \zeta_{N-1}(\tau)+ \mathcal{\dot T}_{il}(\tau) \zeta'_l(\tau)\biggr) \\
&-4 \mathcal{T}^{-1}_{(N-1)i}\bigl(0,z_{N-1}(\xi(\tau)),0\bigr)
a_{ij}(\xi(\tau))
\mathcal{T}^{-1}_{lj}\bigl(0,z_{N-1}(\xi(\tau)),0\bigr)\widehat{\Gamma}_{lm}(\tau)
\zeta'_m(\tau),\quad\tau>t,
\end{aligned}
\end{equation*}
where $(\mathcal{T}^{-1}_{ij}(z))_{i,j=\overline{1,N}}$ is the matrix
inverse  to $(\frac{\partial X_i}{\partial
z_j}(z))_{i,j=\overline{1,N}}$ and
$\mathcal{T}_{ij}(\tau)=\frac{\partial X_i}{\partial
z_j}(0,z_{N-1}(\xi(\tau)),0)$.
Finally we define $\eta(\tau)$ by
$$
\eta(\tau)=
\begin{cases}
X(\zeta^\prime(\tau), z_{N-1}(\xi(\tau))+\zeta_{N-1}(\tau), 0),\
t\leq\tau<T
\\
X(\zeta^\prime(T)(T+1-\tau), z_{N-1}(\xi(\tau))+\zeta_{N-1}(T)
|\dot\xi(\tau)|/|\dot\xi(T)|, 0),\ T\leq\tau<T+1,
\end{cases}
$$
with $T:=\frac{1}{\delta}\log\frac{1}{|z^\prime|}$, and the
control $v(\tau)$ by
\begin{equation*}
v(\tau)=
\begin{cases}
\dot\eta+\nu(\eta)\bigl(b_\nu(\eta)+ R_1(\tau)\bigr),\ \ t\leq\tau<T\\
\dot\eta+\nu(\eta)b_\nu(\eta),\ \ T\leq\tau\leq T+1,
\end{cases}
\end{equation*}
where $z_{N-1}=z_{N-1}(\xi(\tau))$, and 
$$
R_1(\tau)=
-\mathcal{T}^{-1}_{(N-1)i}\bigl(0,z_{N-1},0\bigr) \Big( \mathcal{\dot T}_{i(N-1)}(\tau) \zeta_{N-1}(\tau)+ \mathcal{\dot T}_{il}(\tau) \zeta'_l(\tau)
+4a_{ij}(\xi(\tau))
\mathcal{T}^{-1}_{lj}(0,z_{N-1},0)\widehat\Gamma_{lm}\zeta^\prime_m\Big).
$$
Letting
$$
\alpha(\tau):=\begin{cases}
b_\nu(\eta)+R_1(\tau),\ \ t\leq\tau<T \\
b_\nu(\eta),\ \ T\leq\tau\leq T+1,
\end{cases}
$$
observe that for this control $v(\tau)$
the pair  $(\eta(\tau),\, \alpha(\tau))$
solves (\ref{Skor}) on $(t,T+1)$ with the initial value
$\eta(t)=x(=X(z^\prime,\xi(t),0))$, as far as $\alpha(\tau)\geq 0$ for all
$\tau\in (t,T+1)$. Since $b_\nu(\xi(\tau))>0$, the latter condition is satisfied,
 provided that $|z^\prime|$ is
sufficiently small. Then the proof of (\ref{OTSENKA}) follows exactly
the line of Lemma \ref{Lemupperb}.

In the case when $z_N(x)>0$ we construct a curve $\eta(\tau)$
connecting $x$ with a point $y$ on $\partial \Omega$ by setting
$$
\eta(\tau)=X(z^\prime,z_{N-1}(\xi(t)),z_N+b_\nu(\xi(t))(t-\tau)) \
\text{for all $\tau\geq t$ such that}\ z_N+b_\nu(\xi(t))(t-\tau)\geq
0.
$$
Let $t+\Delta t$ be the time when  $\eta(\tau)$ reaches $\partial \Omega$
(at the point $y=\eta(t+\Delta t)$) then  $\Delta t=O(z_N)$. It follows from the
construction of $\eta(\tau)$ that
$$
\int_t^{t+\Delta
t}a^{ij}(\eta(\tau))(-\dot\eta_i+b_i(\tau))(-\dot\eta_j+b_j(\eta))\,{\rm
d}\tau\leq C (|z^\prime|^2+z_N^2)z_N.
$$
Then extending $\eta(\tau)$ along $\partial\Omega$ as described
above we complete the proof of the Lemma.
\end{proof}

Now we construct test functions
$W^{\pm}_\delta(z^\prime,z_{N-1}(\xi(t)),z_N) :=(\widehat\Gamma\pm
\delta \widehat D)_{ij}(t)z_i^\prime z_j^\prime\pm \delta z_N^2$ for
(sufficiently small) $\delta>0$, where the $P$-periodic symmetric
matrix $\hat D(t)>0$ is defined analogously  to (\ref{defD}). Then
\begin{equation*}
S(\nabla_z W^{-}_\delta,z)\leq -\delta (|z^\prime|^2+b_\nu(\xi(t))z_N)
\end{equation*}
for sufficiently small $|z^\prime|$ and $z_N\geq 0$.
This yields the following bound
$$
W^{-}_\delta<W(X(z)) \ \ \text{for sufficiently small}  \ \ |z^\prime| \ \ \text{and}\ \ z_N \ \text{(when $|z^\prime|+z_N>0$)}
$$
whose proof is analogous to that  of the lower bound in Lemma
\ref{Lcompar}. Thus functions $W^{\pm}_\delta$ satisfy the conditions
of Lemma \ref{MainTechLem}. Finally, we define the test functions
$W^{\pm}_{\delta,\ve}$ by
$$
W^{\pm}_{\delta,\ve}:=W^{\pm}_\delta\pm\ve\widehat\Phi^{\pm}_\delta\mp\ve^2z_N,\
\text{where}\ z_{N-1}=z_{N-1}(\xi(t)),
$$
with $\widehat\Phi^{\pm}_\delta$ being solutions of
\begin{equation*}
\frac{\rm d}{{\rm d} t}\widehat{\Phi}^{\pm}_\delta(t)=-2{\rm
tr}(\widehat{Q}(t)\widehat{\Gamma}^{\pm}_\delta(t))+c(\xi(t))+\frac{2}{P}\int\limits_0^P{\rm
tr}(\widehat{Q}(\tau)\widehat{\Gamma}^{\pm}_\delta(\tau)){\rm d}\tau-
\frac{1}{P}\int\limits_0^P c(\xi(\tau)){\rm d}\tau.
\end{equation*}
These functions $W^{\pm}_{\delta,\ve}$ satisfy the conditions of Lemma \ref{MainTechLem}.

\bigskip\noindent
{\bf Acknowledgements.}\ This work was completed during
the visit of V.~Rybalko at the Narvik University College.
He is indebted for the kind hospitality and financial support.

\bigskip


\bigskip
\begin{thebibliography}{30}


\bibitem{AFIJ}
Abou-Kandil, H., Freiling, G., Ionescu, V.,  Jank, G.
Matrix Riccati equations.
In control and systems theory. Systems and Control: Foundations and Applications. Birkh\"auser Verlag, Basel, 2003.















\bibitem{DeFr} Devinatz, A.; Ellis, R.; Friedman, A. The asymptotic behavior of the first
real eigenvalue of second order elliptic operators with a small parameter in the highest derivatives. II.
{\sl Indiana Univ. Math. J.} {\bf 23} (1973--1974), 991--1011.

\bibitem{Ki87} Eizenberg, A.; Kifer, Yu. The asymptotic
behavior of the principal eigenvalue in a singular perturbation
problem with invariant boundaries. {\sl Probab. Theory Related
Fields},  {\bf 76}(4), (1987),  439--476.

\bibitem{EI} Evans, L. C.; Ishii, H. A PDE approach to some asymptotic problems concerning random differential equations with small noise intensities. {\sl Ann. Inst. H. Poincare Anal. Non Lineaire} {\bf 2}(1) (1985),  1--20.


\bibitem{WF} Freidlin,~M.I., Wentzell,~A.D.
Random perturbations of dynamical systems.
Fundamental Principles of Mathematical Sciences, 260. Springer-Verlag, New York, 1984.


\bibitem{Fri73} Friedman, A. The asymptotic behavior of the first real eigenvalue of a second order
elliptic operator with a small parameter in the highest derivatives. {\sl Indiana Univ. Math. J.} {\bf 22}  (1972/73), 1005--1015.


\bibitem{I} Ishii, H.
Weak KAM aspects of convex Hamilton-Jacobi equations with Neumann
type boundary conditions. {\sl J. Math. Pures Appl. } {\bf 95}(1) (2011), 99--135.


\bibitem{Kif1_80} Kifer, Yu. On the principal eigenvalue in a singular
perturbation problem with hyperbolic limit points and circles.
{\sl J. Differential Equations}, {\bf 37}(1), (1980), 108--139.

\bibitem{Kif2_80} Kifer, Yu., Stochastic stability of the
topological pressure. {\sl J. Analyse Math.}, {\bf 38}, (1980),
255--286.



\bibitem{Ki90} Kifer, Yu.
Principal eigenvalues, topological pressure, and stochastic stability of equilibrium states.
{\sl Israel J. Math.} {\bf 70}(1) (1990), 1--47.

\bibitem{LU}
Ladyzenskaja, O. A.; Ural'ceva, N. N. Certain classes of
nonuniformly elliptic equations. {\sl Zap. Naucn. Sem.
Leningrad. Otdel. Mat. Inst. Steklov. (LOMI)} {\bf 5} (1967), 186--191.


\bibitem{LR} Lancaster, L.; Rodman, L.,
Algebraic Riccati equations. Oxford Science Publications. The
Clarendon Press, Oxford University Press, New York, 1995.



\bibitem{L} Lions, P.-L.
Resolution de problemes elliptiques quasilineaires. {\sl Arch.
Rational Mech. Anal.} {\bf 74}(4) (1980), 335--353.



\bibitem{Pi}Piatnitski, A. Asymptotic Behaviour of the Ground State of Singularly
Perturbed Elliptic Equations. {\sl Commun. Math. Phys.} {\bf 197} (1998),
527--551.

\bibitem{Pe90} Perthame, B. Perturbed dynamical systems with an attracting singularity
and weak viscosity limits in Hamilton-Jacobi equations. {\sl Trans.
Amer. Math. Soc.} {\bf 317}(2) (1990), 723--748.

\bibitem{PR} Piatnitski, A.; Rybalko, V. On the first eigenpair
of singularly perturbed operators with oscillating coefficients.
Submitted to Comm. Part. Diff. Eq.


\bibitem{PRR} Piatnitski, A.; Rybalko, A.; Rybalko, V.
Ground states of singularly perturbed convection-diffusion
equation with oscillating coefficients. {\sl ESAIM: COCV} {\bf 20}(4) (2014) ,
1059--1077.





















\bibitem{S} Serrin, J.
The problem of Dirichlet for quasilinear elliptic differential
equations with many independent variables. {\sl Philos. Trans. Roy.
Soc. London Ser. A} {\bf 264} (1969), 413--496.


\bibitem{We72}Ventcel', A.D.
The asymptotic behavior of the largest eigenvalue of a second order elliptic differential operator with a small parameter multiplying the highest derivatives. (Russian)
{\sl Dokl. Akad. Nauk SSSR} {\bf 202} (1972), 19--22.

\bibitem{We75} Ventcel',~A.D.
The asymptotic behavior of the first eigenvalue of a second order differential operator with a small parameter multiplying the highest derivatives.
{\sl Teor. Verojatnost. i Primenen.} {\bf 20}(3) (1975), 610--613.

\bibitem{ViLu} Vishik, M.I.,  Lyusternik, L.A. Regular
degeneracy and boundary layer for linear differential
equations with a small parameter, {\it Usp. Mat. Nauk} {\bf 12}(5)  (1957), 3--122.

\end{thebibliography}
\end{document}